\documentclass[11pt]{article}

\usepackage[margin=1in]{geometry}

\pdfoutput=1
\usepackage{amsmath}
\usepackage{amsfonts}
\usepackage{amssymb}
\usepackage{graphicx}
\usepackage{algorithm,algpseudocode}
\usepackage{amsthm}
\usepackage{mathdots}
\usepackage[usenames,dvipsnames]{xcolor}
\usepackage{tikz}
\usepackage{color}
\usepackage{hyperref}
\hypersetup{colorlinks=true,
    linkcolor=BrickRed,
    citecolor=OliveGreen,
    urlcolor=MidnightBlue}

\usepackage{booktabs}
\usepackage{algorithmicx}
\usepackage{algpseudocode}
\usepackage{longtable}

\usetikzlibrary{arrows}
\usetikzlibrary{patterns}
\usetikzlibrary{matrix}

\theoremstyle{plain}
\newtheorem{thm}{Theorem}[section]
\newtheorem{proposition}[thm]{Proposition}

\newtheorem{theorem}[thm]{Theorem}
\newtheorem{lemma}[thm]{Lemma}

\newtheorem{corollary}[thm]{Corollary}

\theoremstyle{definition}
\newtheorem{definition}[thm]{Definition}

\newcommand{\RR}{\mathbb{R}}

\newcommand{\cS}{\mathcal{S}}

\newcommand{\cE}{\mathcal{E}}

\renewcommand{\phi}{\varphi}

\newcommand{\psd}{\succeq}
\newcommand{\pd}{\succ}
\newcommand{\nsd}{\preceq}
\newcommand{\nd}{\prec}
\newcommand{\tr}{\textup{tr}}

\newcommand{\conv}{\textup{conv}}
\newcommand{\cone}{\textup{cone}}

\newcommand{\closure}[1]{\textup{cl}(#1)}
\newcommand{\interior}[1]{\textup{int}(#1)}

\newcommand{\relint}{\textup{relint}}

\newcommand{\spn}{\textup{span}}

\newcommand{\sparse}{\textsc{sp}}

\makeatletter
\def\ALG@special@indent{%
    \ifdim\ALG@thistlm=0pt\relax
        \hskip-\leftmargin
    \else
        \hskip\ALG@thistlm
    \fi
}
\newcommand{\Input}[1]{\item[]\noindent\ALG@special@indent \textbf{Input:}\ #1}
\newcommand{\Output}[1]{\item[]\noindent\ALG@special@indent \textbf{Output:}\ #1}
\newcommand{\Indent}[1]{\item[]\noindent\ALG@special@indent \hspace{2.675em} #1}
\makeatother

\title{Sparsification of sums with respect to convex cones}
\author{James Saunderson\thanks{Department of Electrical and Compuater Systems Engineering, Monash University, Clayton VIC 3800, Australia. Email: \texttt{james.saunderson@monash.edu}}}

\begin{document}
\maketitle
\begin{abstract}
	Let $x_1,x_2,\ldots,x_m$ be elements of a convex cone $K$ such that their sum, $e$, is in the relative interior of $K$. An $\epsilon$-sparsification of the sum involves taking a subset of the $x_i$ and reweighting them by positive scalars, so that the resulting sum is $\epsilon$-close to $e$, where error is measured in a relative sense  with respect to the order induced by $K$. 	This generalizes the influential spectral sparsification model for sums of positive semidefinite matrices.

	This paper introduces and studies the sparsification function of a convex cone, which measures, in the worst case over all possible sums from the cone, the smallest size of an $\epsilon$-sparsifier.  The linear-sized spectral sparsification theorem of Batson, Spielman, and Srivastava can be viewed as a bound on the sparsification function of the cone of positive semidefinite matrices.  This result is generalized to a family of convex cones (including all hyperbolicity cones) that admit a $\nu$-logarithmically homogeneous self-concordant barrier with certain additional properties.  For these cones, the sparsification function is bounded above by $\lceil4\nu/\epsilon^2\rceil$.  For general convex cones that only admit an ordinary $\nu$-logarithmically homogeneous self-concordant barrier, the sparsification function is bounded above by $\lceil(4\nu/\epsilon)^2\rceil$. Furthermore, the paper explores how sparsification functions interact with various convex geometric operations (such as conic lifts), and describes implications of sparsification with respect to cones for certain conic optimization problems.
\end{abstract}

\section{Introduction}
\label{sec:intro}

The starting point for this paper is the celebrated linear-sized spectral sparsification 
result of Batson, Spielman, and Srivastava~\cite{batson2012twice}, in the general rank form given in~\cite{silva2015sparse}. This result 
describes how a collection of $d\times d$ positive semidefinite matrices (elements of $\cS_+^d$) that sum to a positive definite matrix (i.e., element of 
$\interior{\cS_+^d}$), can be approximated in a certain spectral sense by keeping only a small number of terms in the sum and suitably reweighting them. 
Here, and throughout, we use the notation $X\nsd Y$ to mean that $Y-X$ is positive semidefinite.
\begin{theorem}
	\label{thm:bss}
	Let $X_1,X_2,\ldots,X_m\in \cS_+^d$ be such that $\sum_{i=1}^{m}X_i = E \in \interior{\cS_{+}^d}$ 
	and let $0<\epsilon<1$. Then there exists a subset $S\subseteq [m]$ with 
	$|S| \leq \lceil 4d/\epsilon^2\rceil$ and positive scalars $(\lambda_i)_{i\in S}$ such that 
	\[ (1-\epsilon)E \nsd \sum_{i\in S} \lambda_i X_i \nsd (1+\epsilon)E.\]
\end{theorem}
This result was motivated by the study of graph sparsification. Indeed a weighted graph $G = (V,E,w)$ with vertex set $V$, edge set $E$, and edge weights $(w_e)_{e\in E}$, 
can be represented by its weighted Laplacian 
\[ L_G = \sum_{\{i,j\}\in E} w_{\{i,j\}}(e_i-e_i)(e_i-e_j)^\intercal.\] 
If $E'\subseteq E$ and $(w'_e)_{e\in E'}$ is another weight function, a sparsification of $L_G$ is a sum of the form
\[ L_{G'} = \sum_{\{i,j\}\in E'} w_{\{i,j\}}'(e_i-e_j)(e_i-e_j)^\intercal\] 
which can be interpreted as the weighted Laplacian of a new weighted graph $G' = (V,E',w')$ with fewer edges. If $L_G$ and $L_{G'}$ 
satisfy $(1-\epsilon)L_{G} \nsd L_{G'} \nsd (1+\epsilon)L_G$ for some $0<\epsilon<1$ then many quantitative properties of $G$ are approximated
by the corresponding properties of $G'$. This idea was initially studied specifically in the context of graph cuts by Bencz\'ur and Karger~\cite{benczur1996approximating}. 
The idea of constructing spectral sparsifiers, i.e., those that approximate the Laplacian in this spectral sense, was pioneered by Spielman and Teng~\cite{spielman2004nearly}.
Before~\cite{batson2012twice} the number of edges in the known cut or spectral sparsifiers 
was $\widetilde{O}(|V|/\epsilon^2)$, where the $\widetilde{O}(\cdot)$ hides poly-logarithmic factors 
(see, e.g.,~\cite{benczur1996approximating},~\cite{spielman2008graph}). 
The major breakthrough of~\cite{batson2012twice} was to obtain a linear-sized spectral sparsifier, 
i.e., one with number of edges that is $O(|V|/\epsilon^2)$.

This paper investigates analogues of linear-sized spectral sparsification for 
general closed convex cones. If $K$ is a closed convex cone and 
$x,y\in \textup{span}(K)$ then we use the notation $x \nsd_K y$ to mean that $y-x\in K$.
If $K$ is, in addition, pointed (i.e., $K \cap (-K) = \{0\}$) then $\nsd_K$ is a partial order 
on $\textup{span}(K)$. The following definitions introduce concise language
to help keep track of the extent to which arbitrary sums of elements from a convex cone 
can be sparsified to within a given error (measured relative to the cone). 

\begin{definition}
	\label{def:sp-constant} Let $K$ be a closed convex cone.
	A \emph{sparsification function for $K$} is a monotonically non-increasing upper semi-continuous 
	function $\alpha:(0,1)\rightarrow \RR$ such that 
	for any $x_1,x_2,\ldots,x_m \in K$ satisfying $\sum_{i=1}^{m}x_i = e\in \relint(K)$
	and any $0<\epsilon<1$, there exists $S\subseteq [m]$ with $|S|\leq  \alpha(\epsilon)$
	and positive scalars $(\lambda_i)_{i\in S}$, such that
	\begin{equation}
		\label{eq:approx}(1-\epsilon)e \nsd_K \sum_{i\in S}\lambda_i x_i \nsd_K (1+\epsilon)e.
	\end{equation}
\end{definition}
\begin{definition}
	\label{def:sp-fn}
	Let $K$ be a closed convex cone.
Let $\mathcal{F}_{K}$ denote the collection of all sparsification functions for $K$. 
	The \emph{sparsification function of $K$}, denoted $\sparse_K:(0,1)\rightarrow \RR$, is defined by
	\[ \sparse_K(\epsilon) = \inf_{\alpha\in \mathcal{F}_K} \alpha(\epsilon).\]
\end{definition}
Note that the pointwise infimum of monotonically non-increasing upper semi-continuous functions is again a monotonically non-increasing upper semi-continuous function. Therefore, $\sparse_K$ is a well-defined sparsification function for $K$. 

It follows directly from Carath\'{e}odory's theorem for cones that $\sparse_K(\epsilon) \leq \dim(K)$ for all $\epsilon\in (0,1)$. (See Lemma~\ref{lem:dim-bound} in Section~\ref{sec:basics}.)  Sparsification functions are 
interesting if, for a non-trivial range of $\epsilon$, they give bounds that are smaller than the dimension of the cone. In such cases 
there is a meaningful trade-off between the approximation error (quantified by $\epsilon$) and the amount that conic 
 combinations of points can be sparsified.

The linear-sized spectral sparsification Theorem (Theorem~\ref{thm:bss}), when stated in the language of Definition~\ref{def:sp-fn}, says that the sparsification function 
of $\cS_+^d$ satisfies $\sparse_{\cS_+^d}(\epsilon) \leq \lceil 4d/\epsilon^2\rceil$. If  $\sqrt{\frac{8}{d+1}} < \epsilon < 1$ then $\sparse_{\cS_+^d}(\epsilon)$
is smaller than the trivial bound of $\dim(\cS_+^d) = \binom{d+1}{2}$ from Carath\'{e}odory's theorem.  

\subsection{Contributions}

This paper develops methods to give upper bounds on sparsification functions of closed convex cones, 
with a focus on bounds that are potentially much smaller than the dimension of the cone. While the results are constructive, 
the emphasis is on the existence of constructions rather than efficient algorithms.
The most general result in this direction applies to arbitrary closed, pointed, full-dimensional 
convex cones (which we refer to as \emph{proper cones}) 
and is the subject of Section~\ref{sec:sc}. 
It gives a bound on the sparsification function that depends only on the barrier parameter 
of a logarithmically homogeneous self-concordant barrier for the cone. 
(See Section~\ref{sec:prelim-sc} for definitions of these terms.)
\begin{theorem}
	\label{thm:gen-sp}
	If $K$ is a proper cone with a $\nu$-logarithmically
	homogeneous self-concordant barrier then $\sparse_K(\epsilon) \leq \lceil (4\nu/\epsilon)^2 \rceil$ for all $\epsilon\in (0,1)$. 
\end{theorem}
Section~\ref{sec:sc-egs} considers examples where this bound on the sparsification function is independent of the dimension of $K$, and so significantly improves over the bound from Carath\'eodory's theorem. However, in the case of the positive semidefinite cone, it is known that $d$ 
is the smallest possible parameter for a logarithmically
homogeneous self-concordant barrier for $\cS_+^d$~\cite{nesterov}. Therefore, when
specialized to the positive semidefinite cone, Theorem~\ref{thm:gen-sp}
cannot improve upon the trivial dimension-based bound from Carath\'{e}odory's theorem, 
let alone match Theorem~\ref{thm:bss}.

The next contribution of this paper, and the focus of Section~\ref{sec:hyp-sparse}, 
is a bound on the sparsification function of 
proper cones that admit a $\nu$-logarithmically-homogeneous self-concordant barrier 
with certain additional properties. Such barriers will be called \emph{pairwise-self-concordant} barriers in this paper, see Definition~\ref{def:ssc}.
Remarkably, for pairwise-self-concordant barriers there is a bound on the sparsification function
that is linear in $\nu$, the barrier parameter.
This result directly generalizes Theorem~\ref{thm:bss}, and is established by suitably generalizing 
the barrier method originally developed in~\cite{batson2012twice}. 
\begin{theorem}
	\label{thm:gen-sp}
	If $K$ is a proper cone with a $\nu$-logarithmically
	homogeneous pairwise-self-concordant barrier then $\sparse_K(\epsilon) \leq \lceil 4\nu/\epsilon^2\rceil$ for all $\epsilon\in (0,1)$. 
\end{theorem}
Hyperbolic polynomials are multivariate polynomials with real coefficients that enjoy certain
real-rootedness properties (see Section~\ref{sec:prelim-hyp-poly} for the definition). Associated
with a hyperbolic polynomial is a convex cone called a hyperbolicity cone. 
Hyperbolicity cones include the positive 
semidefinite cone as a special case. If $p$ is a hyperbolic polynomial of degree $d$, 
then it is well-known that $-\log p$ is a $d$-logarithmically homogeneous self-concordant barrier for the associated hyperbolicity cone~\cite{guler1997hyperbolic}. 
We will see that such hyperbolic barriers are, in fact, pairwise-self-concordant. As a corollary we obtain a generalization of Theorem~\ref{thm:bss} for hyperbolicity cones.
\begin{corollary}
	\label{thm:hyp-sparse}
	Suppose that $p$ is a hyperbolic polynomial of degree $d$, hyperbolic with respect to $e$. 
	Let $\Lambda_+$ denote the associated hyperbolicity cone. Then $\sparse_{\Lambda_+}(\epsilon) \leq \lceil 4d/\epsilon^2\rceil $.
\end{corollary}
By taking the hyperbolic polynomial $p$ to be the determinant restricted to $d\times d$ symmetric matrices (a polynomial of degree $d$), the associated hyperbolicity cone is the positive semidefinite cone $\cS_+^d$. Theorem~\ref{thm:bss} is, therefore, a special case of Corollary~\ref{thm:hyp-sparse}.

Section~\ref{sec:basics} considers how some basic convex geometric operations affect sparsification functions. (See Section~\ref{sec:prelim-cvx-geom} for notation and terminology related to convex geometry.) 
The main result in this direction is that sparsification functions interact nicely with the following
notion of lifted representations (or \emph{extended formulations}) of convex cones~\cite{gouveia2013lifts,fawzi2022lifting}.
\begin{definition}
	\label{def:lift}
	Let $C$ and $K$ be closed convex cones. We say that $C$ has a \emph{$K$-lift} if 
	there is a linear subspace $L\subseteq \textup{span}(K)$ and a linear map 
	$\pi:\textup{span}(K)\rightarrow \textup{span}(C)$ such that $C = \pi(L\cap K)$. If, in addition,
	$L \cap \relint(K)$ is non-empty, we say that $C$ has a \emph{proper $K$-lift}.
\end{definition}
The following result is established in Section~\ref{sec:basics}. 
\begin{proposition}
	\label{prop:lifts}
	Let $C$ and $K$ be closed convex cones.
	If $C$ has a proper $K$-lift then $\sparse_C(\epsilon) \leq \sparse_K(\epsilon)$ for all $\epsilon\in (0,1)$. 
\end{proposition}
The main difficulty in establishing Proposition~\ref{prop:lifts} is to show that if $C = \pi(K)$ where $C$ and $K$ are closed convex cones and $\pi$ is a linear map, then the sparsification function of $C$ is bounded above by the sparsification function of $K$ (Lemma~\ref{lem:proj}).

It is natural to consider whether Proposition~\ref{prop:lifts} can be used in combination with either 
Theorem~\ref{thm:gen-sp} (for general convex cones) or Corollary~\ref{thm:hyp-sparse} (for hyperbolicity cones) 
to obtain further results. If $C$ has a proper $K$-lift then any $\nu$-logarithmically homogeneous self-concordant barrier for $K$ can be used to construct a $\nu$-logarithmically homogeneous self-concordant barrier for $C$ via the implicit barrier theorem~\cite[Theorem 5.3.6]{nesterov}. Therefore, combining Proposition~\ref{prop:lifts} with Theorem~\ref{thm:gen-sp} does not yield any further results. However, Proposition~\ref{prop:lifts} can be used profitably in combination with Corollary~\ref{thm:hyp-sparse}, to give the following result.
\begin{corollary}
	\label{cor:hyp-lift}
	Suppose that $p$ is a hyperbolic polynomial of degree $d$, hyperbolic with 
	respect to $e$ with associated 
	hyperbolicity cone $\Lambda_+$. 
	Let $C$ be a closed convex cone. If $C$ has a $\Lambda_+$-lift then $\sparse_C(\epsilon) \leq \lceil 4d/\epsilon^2\rceil $ for all $\epsilon\in (0,1)$. 
\end{corollary}
This result is interesting because, in general, there are convex cones that have a $\Lambda_+$-lift for some hyperbolicity cone but that are not, themselves, hyperboliticy cones. One way to see this is to observe that hyperbolicity cones are always facially exposed~\cite{renegar2006hyperbolic}, 
but general convex cones with $\cS_+^d$-lifts---let alone hyperbolicity cone lifts---need not be facially exposed~\cite{ramana1995some}.

The astute reader will notice that Corollary~\ref{cor:hyp-lift} is not quite an immediate corollary of Theorem~\ref{thm:hyp-sparse} and Proposition~\ref{prop:lifts}. This is because Corollary~\ref{cor:hyp-lift} only assumes that $C$ has a $\Lambda_+(p,e)$-lift, not a \emph{proper} $\Lambda_+(p,e)$-lift. The proof of Corollary~\ref{cor:hyp-lift} relies on certain properties of the faces of hyperbolicity cones, and is given in Section~\ref{sec:hyp-lift-bounds}.

\subsection{Related work}
Since the work of Bencz\'ur and Karger~\cite{benczur1996approximating} and the spectral sparsification formulation of Spielman and Teng~\cite{spielman2004nearly}, 
many variations and extensions on the problems of cut and spectral sparsification have been studied. This section briefly discusses a selection of such generalizations, focusing
on existence results rather than algorithmic advances. 

Spectral-like sparsification models have been developed for combinatorial objects such as linear codes~\cite{khanna2024code}, hypergraphs~\cite{soma2019spectral}, or constraint satisfaction problems~\cite{khanna2025theory}. For instance, given a hypergraph with $m$ hyperedges $E_1,E_2,\ldots,E_m$ each of which is a subset of $[n]$,
the aim is to find $S\subseteq [m]$ and positive scalars $(\lambda_i)_{i\in S}$ such that
\[ (1-\epsilon)\sum_{i\in [m]} \max_{j,k\in E_i }(x_j - x_k)^2 \leq \sum_{i\in S} \lambda_i \max_{j,k\in E_i }(x_j - x_k)^2 \leq (1+\epsilon)\sum_{i\in [m]} \max_{j,k\in E_i }(x_j - x_k)^2\quad\textup{for all $x\in \RR^n$}.\]
This reduces to the spectral sparsification model for graphs when all hyperedges have degree two, but appears to be 
structurally different to the model in Definition~\ref{def:sp-fn} in general. In this setting spectral hypergraph sparsifiers exist with $|S|$ of size $O(\epsilon^{-2}\log(D)n\log n)$ where $D$ is the maximum size of a hyperedge~\cite{lee2023spectral,jambulapati2023chaining}.

Another family of sparsification problems comes from approximating sums of (semi-)norms on $\RR^n$~\cite{jambulapati2023sparsifying}.
Given a collection $N_1,N_2,\ldots,N_m$ of semi-norms, the aim is to find $S\subseteq[m]$ and positive scalars $(\lambda_i)_{i\in S}$ such that 
\[ (1-\epsilon) \sum_{i\in [m]}N_i(x) \leq \sum_{i\in S}\lambda_i N_i(x) \leq (1+\epsilon)\sum_{i\in [m]}N_i(x)\;\;\textup{for all $x\in \RR^n$}.\]
This generalizes the problem of approximating a centrally symmetric zonotope with $m$ generators (or, more generally, a zonoid) 
by a centrally symmetric zonotope with a small number of generators (see, e.g.,~\cite{bourgain1989approximation}). This is because the support function 
of the zonotope $[-a_1,a_1] + \cdots + [-a_m,a_m]$ (for $a_1,a_2,\ldots,a_m\in \RR^n$) is $\sum_{i\in [m]}N_i(x)$ where $N_i(x) = |\langle a_i,x\rangle|$. 
For the general problem of sparsifying sums of norms, the set $S$ can be taken to be of size $O(\epsilon^{-2}n\log(n/\epsilon)(\log(n))^{2.5})$~\cite{jambulapati2023sparsifying}. 

Related to spectral sparsification of quadratic forms (or equivalently positive semidefinite matrices) is the Kadison-Singer problem, 
resolved by Marcus, Spielman, and Srivastava~\cite{marcus2015interlacing}. In one of its equivalent forms (see~\cite[Theorem 1.2]{paschalidis2024linear}), 
this says that given $v_1,v_2,\ldots,v_m\in \RR^n$, there is a 
partition of $[m] = T_1\cup T_2$ such that, for $j\in \{1,2\}$, 
\[ \left(\frac{1}{2} - 5\sqrt{\alpha}\right)\left(\sum_{i=1}^{m}v_iv_i^\intercal\right) \nsd \sum_{i\in T_j}v_iv_i^\intercal \nsd \left(\frac{1}{2} + 5\sqrt{\alpha}\right)\left(\sum_{i=1}^{m}v_iv_i^\intercal\right)\]
where $\alpha = \max_{j}v_j^\intercal\left(\sum_{i\in [m]}v_iv_i^\intercal\right)^+ v_j$ is the maximum of the leverage scores of the $v_j$. This result shows that 
the positive semidefinite matrix $\sum_i v_iv_i^\intercal$ can be approximated well by summing over either the terms in $T_1$ or the terms in $T_2$, without any reweighting.
On the other hand, the approximation quality depends on the leverage scores of the vectors, and is only small when all of the vectors contribute in a balanced way to the 
sum. Paschalidis and Zhuang have shown that this result implies linear sized spectral sparsification~\cite[Theorem 1.1]{paschalidis2024linear}, but not with the refined constants of Theorem~\ref{thm:bss}. 
The solution of the Kadison-Singer problem has been generalized to hyperbolicity cones by Br\"and\'en~\cite{branden2018hyperbolic}. It seems likely 
that a similar reduction could be carried 
out to deduce linear-sized spectral sparsification with respect to hyperbolicity cones, yielding a qualitatively similar result to Corollary~\ref{thm:hyp-sparse}, 
but again without the refined constants.

\section{Preliminaries}
\label{sec:prelim}

\subsection{Convex geometry}
\label{sec:prelim-cvx-geom}

If $S\subseteq \RR^n$, we use the notation $\cone(S)$ (resp., $\conv(S)$) to denote the conic hull (resp., convex hull) of $S$, i.e., the collection of all non-negative (resp., convex) combinations of elements of $S$. A convex cone $K\subseteq \RR^n$ is \emph{pointed} if $K \cap (-K) = \{0\}$, and \emph{full-dimensional} if $\spn(K) = \RR^n$. 
If $S\subseteq \RR^n$ we denote its closure with respect to the usual topology by $\closure{S}$ and its relative interior (i.e., its interior taken relative to its affine span) by $\relint{(S)}$. 
We call a closed, pointed, full-dimensional convex cone a \emph{proper cone}. Let $\langle \cdot,\cdot\rangle$ denote a fixed choice of inner product on $\RR^n$, which allows us
to identify $\RR^n$ with its dual space. 
If $S\subseteq \RR^n$ is a set then $S^* = \{y\in \RR^n\;:\; \langle y,x\rangle \geq 0,\;\;\textup{for all $x\in S$}\}$ is a closed convex cone called the \emph{dual cone} of $S$.

If $K$ is a closed convex cone, then a convex subset $F\subseteq K$ is a \emph{face} if $x, y \in K$ and $x+y\in F$ implies that $x, y \in F$. A face $F\subseteq K$ is \emph{exposed} if there exists $y\in K^*$ such that $F = \{x\in K\;:\; \langle y,x\rangle = 0\}$, i.e., if $F$ can be expressed as the intersection of $K$ with a supporting hyperplane to $K$.

\subsection{Hyperbolic polynomials and hyperbolicity cones}
\label{sec:prelim-hyp-poly}

A polynomial $p$, homogeneous of degree $d$, in $n$ variables with real coefficients is \emph{hyperbolic with respect to $e\in \RR^n$} if 
$p(e)>0$ and, for every fixed $x\in \RR^n$, if $p(te-x)=0$ then $t$ is real. A fundamental example to keep in mind is the determinant restricted to symmetric $d\times d$ matrices.
This is hyperbolic with respect to the $d\times d$ identity matrix because $p(te-x) = \det(tI - x)$ is the characteristic polynomial of the symmetric matrix $x$, which has only real roots because symmetric matrices have real eigenvalues.

If $p$ is hyperbolic with respect to $e$ and $x\in \RR^n$ then we denote by $\lambda_1^{(p,e)}(x)\geq \lambda_2^{(p,e)}(x)\geq \cdots \geq \lambda_d^{(p,e)}(x)$ 
the (real) roots of $p(te-x)=0$, and call these the \emph{hyperbolic eigenvalues of $x$ with respect to $(p,e)$}. 

Associated with any hyperbolic polynomial $p$, and direction of hyperbolicity $e$, is the \emph{hyperbolicity cone}
\[ \Lambda_+(p,e) = \{x\in \RR^n\;:\; \lambda_{i}^{(p,e)}(x) \geq 0\;\;\textup{for all $i=1,2,\ldots,d$}\},\]
which turns out to be a closed convex cone~\cite{gaarding1959inequality}. In the case of the determinant restricted to $d\times d$ symmetric matrices, the hyperbolicity cone $\Lambda_+(\det,I)$ is the cone $\cS_+^d$ of $d\times d$ positive semidefinite matrices. If $p$ is hyperbolic with respect to $e$ then $p$ is also hyperbolic with respect to any  $e'\in \interior{\Lambda_+(p,e)}$, and $\Lambda_+(p,e) = \Lambda_+(p,e')$~\cite{gaarding1959inequality}. 

If $p$ is hyperbolic with respect to $e$, we say that $p$ is \emph{complete} if $p(te-x) = t^dp(e)$ implies that $x=0$ (i.e., the only point with all hyperbolic eigenvalues 
equal to zero is $0\in \RR^n$).
If $p$ is a complete hyperbolic polynomial then it follows from the definitions that 
the associated hyperbolicity cone is pointed.

\subsection{Self-concordant barriers}
\label{sec:prelim-sc}

This section summarizes the basic definitions and facts (following~\cite{nesterov}) about logarithmically homogeneous 
self-concordant barriers for convex cones needed for subsequent developments.

For a $k$-times 
continuously differentiable function $f$ with open domain, a point $x\in \textup{dom}(f)\subseteq \RR^n$, and directions $u_1,u_2,\ldots,u_k\in \RR^n$, define
\[ D^kf(x)[u_1,\ldots,u_k] := \left.\frac{\partial^k}{\partial t_1 \cdots \partial t_k}f(x+t_1u_1 + \cdots + t_ku_k)\right|_{t_1 = \cdots = t_k =0}.\]
Note that $D^kf(x)[u_1,\ldots,u_k]$ is multilinear in the $u_i$ and that $D^kf(x)[u,\ldots,u]$ is 
the $k$th directional derivative of $f$ at $x$ in the direction $u$. 
\begin{definition}
	A three-times continuously differentiable closed convex function 
	$\phi:\textup{dom}(\phi)\rightarrow \RR$ with open domain $\textup{dom}(\phi)$ is \emph{standard self-concordant} if 
	\[ |D^3\phi(x)[u,u,u]| \leq 2 D^2\phi(x)[u,u]^{3/2}\]
	for all $x\in \textup{dom}(\phi)$ and all $u\in \RR^n$.
\end{definition}

\begin{definition}
	Suppose that $\phi$ is a standard self-concordant function and $\textup{cl}(\textup{dom}(\phi))=:K$ 
	is a closed full-dimensional 
	convex cone. We say that $\phi$ is a 
	$\nu$-\emph{logarithmically-homogeneous self-concordant barrier} for
	$K$ if  $\phi(tx) = \phi(x) - \nu\log(t)$ for all $x\in \textup{dom}(\phi)$ and all $t>0$. 
\end{definition}
Suppose that $\phi$ is a $\nu$-logarithmically-homogeneous self-concordant barrier for a closed, full-dimensional, convex cone $K$. Then we can associate with each point $e\in \textup{int}(K)$ the \emph{Hessian (semi-)norm}
\begin{equation}
\label{eq:hess-norm}
\|u\|_e := D^2\phi(e)[u,u]^{1/2}.
\end{equation}
If $K$ is a proper cone (so is, in addition, pointed), then the Hessian semi-norm is non-degenerate~\cite[Theorem 5.1.6]{nesterov}, and so is actually a norm.

For a full-dimensional closed convex cone $K$ and $x\in \interior{K}$, 
for each $u$ we can also define
an analogue of the spectral norm via
	\begin{equation}
		\label{eq:mg}
		|u|_x := \inf\{t\in \RR\;:\; -tx \nsd_K u \nsd_K tx\}.
	\end{equation}
	If $K$ is, in addition, pointed then $|u|_x = 0$ 
	if and only if $u=0$, and so $|\cdot|_x$ defines a 
	norm. 
The following result collects some standard facts about 
logarithmically homogeneous self-concordant functions and these associated norms.
\begin{lemma}
\label{lem:lhsc-prop}
Let $K\subseteq \RR^n$ be a closed 
	full-dimensional convex cone. Let $\phi$ be a $\nu$-logarithmically homogeneous self-concordant barrier for $K$. Then the following properties hold.
\begin{itemize}
	\item[(i)] \cite[Theorem 5.4.3]{nesterov} If $e\in \textup{int}(K)$ then $D\phi(e)[e] = -\nu$.
	\item[(ii)] \cite[Theorem 5.1.14]{nesterov} If $u\in K$ and $e\in \textup{int}(K)$ then $-D\phi(e)[u] \geq \|u\|_e$.
\item[(iii)] \cite[Theorem 5.1.5]{nesterov} If $e\in \textup{int}(K)$ then $\{z\in \RR^n\;:\; \|z-e\|_e \leq 1\} \subseteq K$.
\item[(iv)] If $e\in \interior{K}$ then $\|u\|_e \geq |u|_e$ for all $u$.
\end{itemize}
\end{lemma}
\begin{proof}
	We only establish (iv), since the other results are directly from~\cite{nesterov}. 
	Let $u$ be arbitrary. To establish that $|u|_e \leq \|u\|_e$, it is enough to show that 
	\begin{equation}
		\label{eq:iv-suff}
		-\|u\|_e e\nsd_K u \nsd_K \|u\|_e e.
	\end{equation}
	To establish~\eqref{eq:iv-suff} we observe that 
	\[ \|(e\pm (\|u\|_e+\epsilon)^{-1}u) - e\|_e = \frac{\|u\|_e}{\|u\|_e+\epsilon} < 1\]
	for all $\epsilon > 0$. It follows from (iii) that $(\|u\|_e+\epsilon )e \pm u\in K$
	for all $\epsilon>0$. Taking the limit as $\epsilon\rightarrow 0$, and using the fact that $K$ is closed, establishes~\eqref{eq:iv-suff}.
\end{proof}

An important example of a logarithmically homogeneous 
self-concordant barrier is the logarithmic barrier for a hyperbolicity cone.
\begin{theorem}[{G\"uler~\cite{guler1997hyperbolic}}]
	Let $p$ be a polynomial, homogeneous of degree $d$, that is hyperbolic with respect to $e$.
	Then $-\log p$ is a $d$-logarithmically homogeneous self-concordant barrier for $\Lambda_+(p,e)$. 
\end{theorem}
If $\Lambda_+(p,e)$ is the hyperbolicity cone associated with a hyperbolic polynomial $p$, then we refer to the $\deg(p)$-logarithmically homogeneous self-concordant barrier $-\log p$ as a \emph{hyperbolic barrier} for $\Lambda_+(p,e)$.

\subsection{Operator monotone functions}
\label{sec:prelim-opmono}

Let $J\subseteq \RR$ be an interval and let $\cS_J^d$ denote the collection of $d\times d$ symmetric
matrices with eigenvalues in $J$. Given a function $h:J\rightarrow \RR$, the associated spectral
function $h:\cS_J^d\rightarrow \cS^d$ is defined by 
\[ h(X) = \sum_{i=1}^{k} h(\lambda_i) P_{i}\]
where $\lambda_1> \cdots > \lambda_k$ are the distinct eigenvalues of $X$, and each $P_{i}$ is the orthogonal projector onto the eigenspace of $X$ corresponding to the eigenvalue $\lambda_i$. 

A function $h:J\rightarrow \RR$ is \emph{$n$-monotone} for some positive integer $n$ if $X,Y\in \cS_J^n$ and $X \psd Y$ together imply that $h(X)\psd h(Y)$. A function that is $n$-monotone 
for all positive integers $n$ is called \emph{operator monotone}. L\"owner's theorem characterizes operator monotone functions. Indeed, any operator monotone function $h:(0,\infty)\rightarrow \RR$, defined on the positive half-line, can be expressed in the form
\begin{equation}
	\label{eq:op-mono-int}
h(t) = h(1) + \int_{0}^{1}\frac{t-1}{\lambda(t-1)+1}\;d\mu(\lambda)\;\;\textup{for all $t\in (0,\infty)$} 
\end{equation}
where $\mu$ is a positive Borel measure $\mu$ supported on $[0,1]$. Note that
similar integral representations hold for operator monotone functions defined on other intervals.
Moreover, there are many different equivalent ways to write integral representations of operator monotone functions~\cite{simon2019loewner}. For a version of L\"owner's theorem written in the form of~\eqref{eq:op-mono-int}, see~\cite[Theorem 4]{fawzi2019semidefinite}.

\section{Basic results about sparsification functions}
\label{sec:basics}

\subsection{Upper bounds}
\label{sec:upper}

This section establishes some bounds on sparsification functions of  convex cones arising from 
basic convex geometric considerations.

\begin{lemma}
	\label{lem:dim-bound}
	If $K$ is a closed convex cone then $\sparse_K(\epsilon) \leq \dim(K)$ for all $\epsilon\in (0,1)$.
\end{lemma}
\begin{proof}
	Let $x_1,\ldots,x_m\in K$, let $\epsilon\in (0,1)$,  and let $e = \sum_{i=1}^{m}x_i\in \relint(K)$. 
	Let $C$ be the conic hull of the $x_i$, i.e, $C = \cone\{x_1,\ldots,x_m\}\subseteq K$. 
	By Carath\'eodory's theorem,
	there exists $S\subseteq [m]$ with $|S|\leq \dim(C)\leq \dim(K)$ and positive scalars $(\lambda_{i})_{i\in S}$
	such that $\sum_{i\in S}\lambda_i x_i = e$, and so 
	\[ (1-\epsilon)e \nsd_K e = \sum_{i\in S}\lambda_i x_i \nsd (1+\epsilon)e.\]
	It follows that $\dim(K)$ is a sparsification function for $K$, and hence that 
	$\sparse_K(\epsilon) \leq \dim(K)$ for all $\epsilon\in(0,1)$ by the definition of $\sparse_K$. 
\end{proof}
The following result shows that sparsification functions do not increase under taking the intersection 
with a subspace that meets the relative interior of a cone.
\begin{lemma}
	\label{lem:int}
 	Let $K$ be a closed convex cone and let $L$ be a linear subspace such that
	$\relint(K)\cap L\neq \emptyset$. If $C = K \cap L$ then  $\sparse_C(\epsilon) \leq \sparse_K(\epsilon)$ for all $\epsilon\in (0,1)$.
 \end{lemma}
 \begin{proof}
	 Suppose that $x_1,x_2,\ldots,x_m\in C= K \cap L$ and $\sum_{i=1}^{m}x_i = e \in \relint(K\cap L)$ and 
	 $0<\epsilon<1$. 
	 Since $\relint(K)\cap L \neq \emptyset$ it follows that $\relint(K)\cap L = \relint(K \cap L)$~\cite[Corollary 6.5.1]{rockafellar}.
	 Therefore $e\in \relint(K)$. From the definition of $\sparse_K$, 
	 there exists some $S\subseteq [m]$ with $|S|\leq \sparse_K(\epsilon)$ and positive 
	 scalars $(\lambda_i)_{i\in S}$ such that
	 \[ (1-\epsilon)e \nsd_K \sum_{i\in S}\lambda_i x_i \nsd_K (1+\epsilon)e.\]
	 Since $e\in L$ and $\sum_{i\in S}\lambda_i x_i \in L$ it follows that 
	 \[ (1-\epsilon)e \nsd_C \sum_{i\in S}\lambda_i x_i \nsd_C (1+\epsilon)e,\]
	 and so that $\sparse_K(\epsilon)$ is a sparsification function for $C$. The inequality $\sparse_C(\epsilon)\leq \sparse_K(\epsilon)$ then follows from the definition of $\sparse_C$. 
 \end{proof}
 
Sparsification functions also do not increase under linear projections.
 \begin{lemma}
	 \label{lem:proj}
	 Let $K$ be a closed convex cone and let $\pi:\spn(K)\rightarrow \RR^k$
	 be a linear map. Then $\sparse_{\pi(K)}(\epsilon) \leq\sparse_K(\epsilon)$ for all $\epsilon\in (0,1)$. Moreover, if $\pi$ is a linear isomorphism then $\sparse_{\pi(K)}(\epsilon) = \sparse_K(\epsilon)$ for all $\epsilon\in (0,1)$. 
 \end{lemma}
\begin{proof}
	Suppose that $x_1,x_2,\ldots,x_m\in \pi(K)$ and $\sum_{i=1}^{m}x_i = e\in \relint(\pi(K))$ and $0<\epsilon<1$.
	Let $x_1',x_2',\ldots,x_m'\in K$ be such that $\pi(x_i') = x_i$ for $i\in [m]$. Let $e' = \sum_{i=1}^{m}x_i'$. 
	Since $\pi$ is linear, $\pi(e') = e$. Since $x_i'\in K$ for all $i$, it follows that $e'\in K$. 
	There are two cases to consider, depending on whether or not $e'\in \relint(K)$.

	If $e'\in \relint(K)$ then, by the definition of $\sparse_K$, there exists $S\subseteq[m]$ with $|S|\leq  \sparse_K(\epsilon)$
	and positive scalars $(\lambda_i)_{i\in S}$ such that $(1-\epsilon)e' \nsd_K \sum_{i\in S} \lambda_i x_i' \nsd_K (1+\epsilon)e'$. 
	By applying $\pi(\cdot)$ it follows that $(1-\epsilon)e \nsd_{\pi(K)} \sum_{i\in S}\lambda_i x_i \nsd_{\pi(K)} (1+\epsilon)e$, implying that $\sparse_K$ is a sparsification function for $\pi(K)$, as required.

	If $e'$ is in the relative boundary of $K$ then $x_1',x_2',\ldots,x_m'$ and $e'$ all lie in a proper face of $K$, and hence in a proper exposed face of $K$ (due to the existence of a supporting hyperplane to $K$ at $e'$). Therefore
	there is a linear functional $\xi\in K^*$ such that $\xi[x_i'] = 0$ for all $i$ and $\xi[e']=0$.
	Since $\relint(\pi(K)) = \pi(\relint(K))$~\cite[Theorem 6.6]{rockafellar}, there exists $\tilde{e}\in \relint(K)$ such 
	that $\pi(\tilde{e}) = \pi(e') =  e$. Let $\delta>0$, let $x_{m+1}' = \delta\tilde{e}$, and consider the collection of points $x_1',x_2',\ldots,x_m',x_{m+1}'\in K$.
	Note that $\sum_{i=1}^{m+1}x_i' = e'+\delta\tilde{e} =:e'' \in \relint(K)$ since $x_{m+1}' \in \relint(K)$ and $x_i'\in K$ for all $i\in [m]$.
	By the definition of $\sparse_K$ there exists $S\subseteq[m]$ with $|S|\leq \sparse_K\left(\frac{\epsilon}{1+2\delta}\right)$,
	positive scalars $(\lambda_i)_{i\in S}$, and a non-negative scalar $\lambda_{m+1}$, such that 
	\begin{equation}
		\label{eq:K-ineq-proj}
		\left(1-\frac{\epsilon}{1+2\delta}\right)e'' \nsd_K \lambda_{m+1}x_{m+1}' + \sum_{i\in S} \lambda_i x_i' \nsd_K \left(1+\frac{\epsilon}{1+2\delta}\right)e''.
	\end{equation}
	
	To complete the argument, we need to control the value of $\lambda_{m+1}$. Recall that there is a 
	linear functional $\xi\in K^*$ such that $\xi[x_i'] = 0$ for all $i$ and $\xi[e']=0$. Applying this to both sides of~\eqref{eq:K-ineq-proj} gives
	\begin{align*}
		\left(1-\frac{\epsilon}{1+2\delta}\right)\delta \xi[\tilde{e}] = \left(1-\frac{\epsilon}{1+2\delta}\right)\xi[e''] & \leq \lambda_{m+1}\xi[x_{m+1}'] = \delta \lambda_{m+1}\xi[\tilde{e}]\\
		&\leq \left(1+\frac{\epsilon}{1+2\delta}\right)\xi[e''] = 
	\left(1+\frac{\epsilon}{1+2\delta}\right)\delta \xi[\tilde{e}].
	\end{align*}
	Since $\tilde{e}\in \relint(K)$ it follows that $\xi[\tilde{e}]>0$. Therefore, $\left(1-\frac{\epsilon}{1+2\delta}\right)\leq \lambda_{m+1} \leq \left(1+\frac{\epsilon}{1+2\delta}\right)$. 

	Now we will consider the image of~\eqref{eq:K-ineq-proj} under $\pi$. Since $\pi(e'') = \pi(e'+\delta\tilde{e}) = (1+\delta)e$, we have
	\[ \left(1-\frac{\epsilon}{1+2\delta}\right)(1+\delta)e \nsd_{\pi(K)} \lambda_{m+1}\delta e +  \sum_{i\in S}\lambda_i x_i \nsd_{\pi(K)} \left(1+\frac{\epsilon}{1+2\delta}\right)(1+\delta)e.\]
	Using the fact that $\delta\left(1-\frac{\epsilon}{1+2\delta}\right)e \nsd_{\pi(K)} \delta\lambda_{m+1}e \nsd_{\pi(K)} \left(1+\frac{\epsilon}{1+2\delta}\right)\delta e$ gives
	\[ (1-\epsilon)e \nsd_{\pi(K)}\sum_{i\in S}\lambda_i x_i \nsd_{\pi(K)} (1+\epsilon)e.\]
	It follows that $\sparse_K(\epsilon/(1+2\delta))$ is a sparsification function for $\pi(K)$. Since 
	$\delta > 0$ was arbitrary, it follows that $\alpha(\epsilon) := \inf_{\delta>0}\sparse_K(\epsilon/(1+2\delta))$ 
	is a sparsification function for $\pi(K)$. Since sparsification
	functions are (by definition) upper semi-continuous and non-increasing, it follows that 
	$\alpha(\epsilon) = \sparse_K(\epsilon)$, and so $\sparse_K$ is a sparsification function for $\pi(K)$, as required. 

	In the case where $\pi$ is a linear isomorphism, we have that $C$ is the image of $K$ under a linear map and $K$ is also the image of 
	$C$ under a linear map. Therefore $\sparse_{K}(\epsilon) \leq \sparse_{C}(\epsilon) \leq \sparse_K(\epsilon)$ for all $\epsilon\in (0,1)$. 
\end{proof}
Since sparsification functions do not increase under (proper) intersections with a subspace and under projections, it follows that the sparsification functions do not increase under proper lifts (see Definition~\ref{def:lift}), allowing us to establish Proposition~\ref{prop:lifts}. 
\begin{proof}[{Proof of Proposition~\ref{prop:lifts}}]
	This follows directly from Lemmas~\ref{lem:int} and~\ref{lem:proj}.
\end{proof}

We defined the sparsification function for any closed convex cone. The following result shows that
we can, without loss of generality, consider proper cones.

\begin{lemma}
	Let $K$ be a closed convex cone and let $L = K \cap (-K)$ be the lineality space of $K$. 
	Then the sparsification functions of $K$ and the closed pointed cone 
	$K' = K\cap L^\perp$ are equal.
\end{lemma}
\begin{proof}
	If $L = \{0\}$ then $K = K'$ so there is nothing to prove. Otherwise, we have the 
	direct sum decomposition $K = K' + L$~\cite[Page 65]{rockafellar}. 
	Note that if $z\in K'$ and $v\in L$ then $(z,v)\in K$ if and only 
	if $z\in K'$ and $v\in L$. If, in addition $z'\in K'$ and $v'\in L$ 
	then $(z,v) \nsd_K (z',v')$ if and only if $(z-z',v-v')\in K'+ L$ which 
	holds if and only if $z\nsd_{K'} z'$.
	
	Since $K'$ is the image of $K$ under a linear map, 
	$\sparse_{K'}(\epsilon)\leq \sparse_K(\epsilon)$ for all $\epsilon\in (0,1)$. 
	For the reverse inequality, let $(z_1,v_1),\ldots,(z_m,v_m)\in K$ be such that 
	$\sum_{i\in [m]}(z_i,v_i) = (z_0,v_0)\in \relint(K)$. 
	Then, by the definition of $\sparse_{K'}(\epsilon)$ there exists $S\subseteq [m]$ with $|S|\leq \sparse_{K'}(\epsilon)$ and positive scalars $(\lambda_i)_{i\in S}$ such that 
	\[ (1-\epsilon)z_0 \nsd_{K'} \sum_{i\in S}\lambda_i z_i \nsd_{K'} (1+\epsilon)z_0.\]
	But this holds if and only if  
	\[ (1-\epsilon)(z_0,v_0) \nsd_{K} \sum_{i\in S}\lambda_i (z_i,v_i) \nsd_{K} (1+\epsilon)(z_0,v_0),\]
	showing that $\sparse_{K}(\epsilon)\leq \sparse_{K'}(\epsilon)$.
\end{proof}

\subsection{Faces}
\label{sec:faces}

This section considers how sparsification functions interact with the facial structure of a convex cone. It is reasonable to hope that sparsification functions might be monotone along faces, in the sense that if $F\subseteq K$ is a face of $K$ then $\sparse_F(\epsilon)\leq \sparse_K(\epsilon)$ for all $\epsilon\in (0,1)$. However we are only able to establish this result under additional 
technical assumptions on the face $F$. The technical assumptions we involve are related to certain strong notions of facial exposedness. The first of these is the notion of a face being projectionally exposed. This was first introduced by Borwein and Wolkowicz~\cite{borwein1981regularizing} in their study of facial reduction methods for abstract convex optimization problems.
\begin{definition}
	Let $K$ be a closed convex cone. A face $F$ of $K$ is \emph{projectionally exposed} if there is an idempotent linear map $\pi:\spn(K)\rightarrow \spn(K)$ such that 
	$\pi(K) = F$.
\end{definition}
 \begin{lemma}
	 \label{lem:mono-pexp}
	 If $K$ is a closed convex cone and $F$ is a projectionally exposed face of $K$ then $\sparse_F(\epsilon)\leq \sparse_K(\epsilon)$ for all $\epsilon\in (0,1)$. 
 \end{lemma}
 \begin{proof}
 	If $F$ is a projectionally exposed face of $K$ then there exists an idempotent 
	 linear map $\pi$ such that $\pi(K) = F$. The result then follows from Lemma~\ref{lem:proj}.
 \end{proof}
 Another, more recent, strengthening of facial exposedness is the notion of an amenable face, introduced by Louren\c{c}o~\cite{lourencco2021amenable} in the context of developing error bounds for conic optimization problems in the absence of constraint qualification. 
 \begin{definition}
	 Let $K$ be a closed convex cone. A face $F$ of $K$ is \emph{amenable} if
	 there is a positive constant $\kappa$ such that 
	 \[ d(x,F) \leq \kappa\,d(x,K)\;\;\textup{for all $x\in \spn(F)$}.\]
 \end{definition}

 The following result is strictly stronger that Lemma~\ref{lem:mono-pexp}, since it is known that 
 there are cones with faces that are amenable but not projectionally exposed~\cite{LRS2025}. 

 \begin{lemma}
	 \label{lem:mono-amenable}
	 If $K$ is a closed convex cone and $F$ is an amenable face of $K$ then 
	 $\sparse_F(\epsilon)\leq \sparse_K(\epsilon)$ for all $\epsilon\in (0,1)$. 
 \end{lemma}
 \begin{proof}
	 Let $e\in \relint(K)$, let $S = \spn(F \cup\{e\})$  and consider $\hat{K} = K \cap S$.
	If $F = K$ then the result trivially holds. As such, from now on we assume that $F$ is a proper subset of $K$. Therefore $F$ is disjoint from the relative interior
	 of $K$, so $e\notin F$. It follows that $\dim(F) = \dim(S) - 1 = \dim(\hat{K})-1$ (where the last equality holds because $S$ meets the relative interior of $K$). Moreover, $F\cap S = F$ is a face of $\hat{K} = K \cap S$. 

	 By assumption, $F$ is an amenable face of $K$. As such, there exists $\kappa>0$ such that 
	 \[ d(x,F) \leq \kappa\, d(x,K)\;\;\textup{for all $x\in \spn(F)$}.\]
	 Since $\hat{K}\subseteq K$ it follows that $d(x,K) \leq d(x,\hat{K})$ for all $x$. 
	 Therefore,
	 \[ d(x,F) \leq \kappa d(x,K) \leq \kappa d(x,\hat{K})\;\;\textup{for all $x\in \spn(F)$}.\]
	 Hence $F$ is a codimension one amenable face of $\hat{K}$. From~\cite[Theorem 6.2]{lourenco2022amenable}, it follows that $F$ is a projectionally exposed face of $\hat{K}$. From Lemma~\ref{lem:mono-pexp} we see that $\sparse_{F}(\epsilon) \leq \sparse_{\hat{K}}(\epsilon)$. Finally, 
	 since $S$ meets the relative interior of $K$, it follows from Lemma~\ref{lem:int} 
	 that $\sparse_{\hat{K}}(\epsilon) \leq \sparse_K(\epsilon)$ for all $\epsilon\in (0,1)$. 

	 Overall, we have that $\sparse_F(\epsilon) \leq \sparse_K(\epsilon)$, as required.
 \end{proof}


\section{Sparsification and self-concordant barriers}
\label{sec:sc}

This section develops a very simple method for sparsficiation with respect to 
a general proper convex cone $K$. The basic idea is to find a point that is close to 
$e = \sum_{i\in [m]}x_i$ in the appropriate sense by running the Frank-Wolfe algorithm 
on an optimization problem with a quadratic objective and a constraint set with extreme 
points that are scaled versions of the input points $x_1,\ldots,x_m\in K$. Since each 
iteration of the Frank-Wolfe algorithm steps in the direction of an extreme point, the number
of iterations bound the number of terms in the decomposition of 
the current iterate. We use, as the objective, the squared Hessian norm of a $\nu$-logarithmically homogeneous self-concordant
barrier for $K$. Such an objective is related (by part (iv) of Lemma~\ref{lem:lhsc-prop}) 
to the spectral norm $|\cdot|_e$ with respect to $K$, the quantity
we truly would like to control.
By taking the extreme points of the constraint set to be scaled in a way that is 
also informed by the self-concordant barrier, we obtain a bound on the curvature constant 
of the objective that depends only on $\nu$. Combining these ingredients
gives a bound (Theorem~\ref{thm:sc-sparse}) on the sparsification function that depends 
only on $\nu$ and $\epsilon$.

Theorem~\ref{thm:sc-sparse} is not strong enough to 
give a non-trivial sparsification result for the positive semidefinite cone. However, it does
give interesting examples of convex cones for which dimension-independent sparsification is possible. We discuss this further in Section~\ref{sec:sc-egs}.

\subsection{The Frank-Wolfe algorithm}
\label{sec:fw}

The approach we will use to construct a sparsifier with respect to a general convex cone will be based on an algorithm for minimizing 
smooth convex functions over compact convex sets, known as the \emph{Frank-Wolfe} (or \emph{conditional gradient}) algorithm~\cite{fw}.
Let $f$ be a differentiable convex function, and let $\mathcal{X}$ be a compact convex subset of the domain of $f$.
Consider the optimization problem
\[ f^\star := \min_{z\in \mathcal{X}} f(z).\]
The Frank-Wolfe algorithm, given in its most basic form in Algorithm~\ref{alg:fw}, is a method to solve this problem. 

\begin{algorithm}
	\caption{Frank-Wolfe algorithm}
    \begin{algorithmic}
	    \Input{Differentiable convex function $f$, compact convex set $\mathcal{X}\subseteq \textup{dom}(f)$}
	    \State \textbf{Initialize:} Initial point $z_0\in \mathcal{X}$
		\For{$t=0,1,\ldots$}{}
		\State $\alpha_t = \frac{2}{t+2}$
		\State $v_t \in \arg\min_{v\in \mathcal{X}} \langle \nabla f(z_t),v\rangle$
		\State $z_{t+1} = (1-\alpha_t)z_{t} + \alpha_tv_t$
        \EndFor
    \end{algorithmic}
    \label{alg:fw}
\end{algorithm}

Each step of the algorithm involves minimizing a linear functional over $\mathcal{X}$ to obtain
a direction $v_t$. The next iterate is then a convex combination of the current iterate and 
the direction $v_t$. Since linear optimization over a compact convex set always has an optimal 
point that is an extreme point of the set, we can always choose $v_t$ to be an extreme point
of $\mathcal{X}$. If we do this, the Frank-Wolfe algorithm (as stated in Algorithm~\ref{alg:fw})
has the property that the $t$th iterate, $z_t$, is a convex combination of at most $t$ 
extreme points of $\mathcal{X}$. 

The main parameter of the objective function that appears in typical convergence results for the Frank-Wolfe algorithm (see Theorem~\ref{thm:fw-conv}, to follow)
is the \emph{curvature constant} of $f$ with respect to $\mathcal{X}$, defined as 
\[ C_f := \sup_{\substack{z,s\in \mathcal{X},\\\gamma\in [0,1]}} \frac{2}{\gamma^2}B_f((1-\gamma)z+\gamma s||z)\]
where 
$B_f(x||y) := f(x) - f(y) - \langle \nabla f(y),x-y\rangle$
is the \emph{Bregman divergence} associated with $f$. 

\begin{theorem}[{\cite[Theorem 1]{jaggi2013revisiting}}]
	\label{thm:fw-conv}
	Let $f$ be a differentiable convex function with curvature constant $C_f$ and let $\mathcal{X}$ be a compact convex set contained in the domain of $f$.
	Then the iterates $z_t$ of Algorithm~\ref{alg:fw} applied to $f$ and $\mathcal{X}$ satisfy
	\[ f(z_t) - f^\star \leq \frac{2 C_f}{t+2}\quad\textup{for all $t\geq 1$}.\]
\end{theorem}

\subsection{A bound on the sparsification function}
In this section we establish a general bound on the sparsification function of a convex cone in terms of the barrier parameter of a logarithmically homogeneous self-concordant barrier for the cone.

\begin{theorem}
	\label{thm:sc-sparse}
	Let $K\subseteq \RR^n$ be a closed, pointed, full-dimensional convex cone. Let $\phi$ be a logarithmically homogeneous self-concordant barrier for $K$ with parameter $\nu$. Then $\sparse_K(\epsilon)\leq\lceil (4\nu/\epsilon)^2\rceil$.
\end{theorem}

\begin{proof}
First, we will find scalars $\mu_i>0$, and define rescaled points 
$\tilde{x}_i := \mu_i x_i$ for $i=1,2,\ldots,m$, such that $e$ is a convex combination of the $\tilde{x}_i$. 
Given such scalars, if there exists a set $S\subset [m]$ and positive scalars $\tilde{\lambda_i}$ for $i\in S$ such that 
	$(1-\epsilon)e \nsd_K \sum_{i\in S}\tilde{\lambda}_i\tilde{x}_i \nsd_K (1+\epsilon)e$, then 
	Theorem~\ref{thm:sc-sparse} follows by taking $\lambda_i = \tilde{\lambda}_i\mu_i$ for $i\in S$.
	As such, it is enough to 
	work with the $\tilde{x}_i$. 

To find the scalars $\mu_i$, we note that since $\sum_{i=1}^{m}x_i = e$, by property (i) of Lemma~\ref{lem:lhsc-prop} we have that
	\[ -\sum_{i=1}^{m}D\phi(e)[x_i] = -D\phi(e)[e] = \nu.\]
	Let $w_i = -D\phi(e)[x_i]/\nu$ for $i=1,2,\ldots,m$.
	Since $x_i\in K$ for $i=1,2,\ldots,m$, it follows from property (ii) of Lemma~\ref{lem:lhsc-prop} that $w_i\geq\|x_i\|_e/\nu > 0$ for $i=1,2,\ldots,m$ (since $K$ is pointed so $\|\cdot\|_e$ is non-degenerate). 
	Furthermore, $\sum_{i=1}^{m} w_i = 1$. As such, let $\tilde{x}_i = w_i^{-1}x_i$. Then $\tilde{x_i}\in K$ for $i=1,2,\ldots,m$ and $\sum_{i=1}^{m}w_i \tilde{x}_i = \sum_{i=1}^{m}x_i = e$, 
so $e$ is a convex combination of the $\tilde{x}_i$. Therefore we can take $\mu_i = w_i^{-1}$ for $i\in [m]$.

Let $\mathcal{X}:= \textup{conv}\{\tilde{x}_1,\tilde{x}_2,\ldots,\tilde{x}_m\}$ and let $f(z) = \frac{1}{2}\|z-e\|_e^2$. Consider the optimization problem
\[ \min_{z\in \mathcal{X}} f(z).\]
Since $e\in \mathcal{X}$, it follows that the optimal value $f^\star$ of this problem is zero. 
Moreover, since 
	\[ -D\phi(e)[\tilde{x}_i] = \nu\quad\textup{for all $i$,}\]
	and $-D\phi(x)[\cdot]$ is a linear functional, 
	it follows that $-D\phi(e)[z] = \nu$ for all $z\in \mathcal{X}$. 

We next compute the curvature constant of $f$ with respect to $\mathcal{X}$.
	A straightforward computation shows that 
	\[ B_f(x||y) = f(x) - f(y) - \langle \nabla f(y),x-y\rangle = \frac{1}{2}\|x-y\|_e^2.\]
Therefore, the curvature constant can be bounded as
\begin{align*}
C_f & = \sup_{\substack{z,s\in \mathcal{X},\\\gamma\in [0,1]}} \frac{2}{\gamma^2}B_f((1-\gamma)z+\gamma s||z)\\
& = \sup_{\substack{z,s\in \mathcal{X},\\\gamma\in [0,1]}} \frac{1}{\gamma^2}\|\gamma (s-z)\|_e^2\\
& = \sup_{\substack{z,s\in \mathcal{X},\\\gamma\in [0,1]}} \|s-z\|_e^2\\
& \leq \left(2\sup_{z\in \mathcal{X}} \|z\|_e\right)^2\quad\textup{(by the triangle inequality)}\\
	& \leq 4 \sup_{z\in \mathcal{X}}(-D\phi(e)[z])^2\\
& = 4\nu^2
\end{align*}
	where the last equality uses the fact that $-D\phi(e)[z] = \nu$ for all $z\in \mathcal{X}$. 

After $T = \lceil \frac{16 \nu^2}{\epsilon^2}\rceil$ steps of the Frank-Wolfe algorithm, we obtain a point $z_T$ that is a convex combination of at most $T$ elements of $\{\tilde{x}_1,\ldots,\tilde{x}_m\}$. 
Moreover, 
\[ f(z_T) - f^\star = f(z_T) \leq \frac{2 C_f}{T+2} \leq \frac{8 \nu^2}{T} \leq \epsilon^2/2.\]
	Therefore, $\|z_T-e\|_e \leq \epsilon$. From Lemma~\ref{lem:lhsc-prop} part (iv) we have that $|z_T-e|_e \leq \|z_T-e\|_e\leq \epsilon$ and so 
\[ (1-\epsilon)e \nsd_K z_T \nsd_K (1+\epsilon)e,\]
as required.
\end{proof}

\section{Generalizing Batson-Spielman-Srivastava sparsification}
\label{sec:hyp-sparse}

This section shows that if a closed, pointed, full-dimensional convex cone admits a $\nu$-self-concordant barrier with certain additional properties (see Definition~\ref{def:ssc}), then the sparsification function is bounded by an expression that is \emph{linear} in the barrier parameter $\nu$. 
This is an exact generalization of the linear-sized spectral sparsification result of Batson, Spielman, and Srivastava (Theorem~\ref{thm:bss}). 
\begin{theorem}
	\label{thm:main-sc-gen}
	Let $K$ be a closed, pointed, full-dimensional, convex cone. 
	If $K$ has a $\nu$-logarithmically homogeneous pairwise-self-concordant barrier 
	then $\sparse_K(\epsilon) \leq \lceil 4\nu/\epsilon^2\rceil $. 
\end{theorem}
The additional `pairwise-self-concordance' property constrains the change of the barrier function 
along \emph{pairs} of directions in the cone $K$, whereas the usual notion of self-concordance 
is a constraint that relates the second and third derivatives along arbitrary lines.
\begin{definition}
	\label{def:ssc}
	Let $K$ be a closed, pointed, full-dimensional convex cone. 
	A logarithmically homogeneous self-concordant barrier $\phi$ for $K$ 
	is \emph{pairwise-self-concordant} if,
	for all $u,v\in K$ and all $x\in \interior{K}$, 
	\begin{equation}
0 \leq 	-D^3\phi(x)[v,u,u] \leq 2D^2\phi(x)[v,u]|u|_x.
	\end{equation}
\end{definition}
It seems that this additional property has not been considered before, so we discuss it in more 
detail in Section~\ref{sec:ssc}. In particular, we give a sufficient condition that implies pairwise-self-concordance (see Lemma~\ref{lem:mono-cond}).
The $d$-logarithmically homogeneous self-concordant 
barrier $-\log\det(\cdot)$ for the positive semidefinite cone satisfies this sufficient condition, and so is pairwise-self-concordant (see Lemma~\ref{lem:psd-mono} in Section~\ref{sec:ssc}).
Therefore Theorem~\ref{thm:main-sc-gen} exactly generalises Theorem~\ref{thm:bss}. 
Furthermore, hyperbolic barriers also satisfy the sufficient condition for being pairwise-self-concordant. This tells us that Theorem~\ref{thm:main-sc-gen} implies Theorem~\ref{thm:hyp-sparse}, 
our main sparsification result for hyperbolicity cones. 

The aim of the rest of this section is to establish Theorem~\ref{thm:main-sc-gen}, by appropriately generalizing the deterministic sparsification argument of Batson, Spielman, and Srivastava~\cite{batson2012twice}.

\subsection{Barrier functions and the generalized BSS algorithm}
\label{sec:bss-alg}
The sparsification method of Batson, Spielman, and Srivastava applies to the positive semidefinite cone and the problem of sparsifying a collection $(X_i)_{i\in [m]}$ of positive semidefinite matrices that sum to the identity, i.e.,  $\sum_{i\in [m]}X_i = I$. 

If $\ell\,I \nd X \nd u\,I$ then they define upper and lower barrier functions
\[ \Phi^u(X) = \textup{tr}\left[(uI-X)^{-1}\right]\;\;\textup{and}\;\;\Phi_\ell(X) = \textup{tr}\left[(X-\ell I)^{-1}\right].\]
Their algorithm uses these barrier functions as a guide to control 
the eigenvalues of the current iterate and to take steps in the direction of $X_i$ for some $i\in [m]$, in such a way that the eigenvalues of the new matrix become more uniform.

In the general setting, we note that $\textup{tr}(X^{-1}) = -D\log\det(X)[I]$. Therefore, 
a natural generalization of the upper and lower barrier functions is
to take
\begin{align}
	\Phi^{u,e}(x) &= -D\phi(ue-x)[e]\quad\textup{and}\quad\label{eq:ubarrier-def}\\
	\Phi_{\ell,e}(x) &= -D\phi(x-\ell e)[e].\label{eq:lbarrier-def}
\end{align}
This choice leads to what 
we call the generalized BSS algorithm (Algorithm~\ref{alg:gen-bss}), which is the exact analogue of the original BSS algorithm~\cite[Algorithm]{batson2012twice} in the abstract conic setting.
\begin{algorithm}
	\caption{Generalized BSS algorithm for a proper cone $K$ that admits a $\nu$-logarithmically homogeneous pairwise-self-concordant barrier $\phi$ with associated upper and lower barrier functions defined in~\eqref{eq:ubarrier-def} and~\eqref{eq:lbarrier-def}.}
    \begin{algorithmic}
	    \Input{points $x_1,x_2,\ldots,x_m\in K$ such that $\sum_{i\in [m]}x_i =:e \in \interior{K}$; $0<\epsilon<1$}
	    \Output{$\lambda \geq 0$  with at most $T$ non-zero entries satisfying  
	    $(1-\epsilon)e \nsd_K \sum_{i\in [m]}\lambda_i x_i \nsd_K (1+\epsilon)e$}.
	    \State \textbf{Initialize:} $z_0 = 0\in \spn(K)$; $y_0 = 0\in \RR^m$; $T = \lceil 4\nu/\epsilon^2\rceil$;\\\hspace{1cm} $\delta_u = 1+\epsilon/2$; $\delta_\ell = 1-\epsilon/2$, $u_0 = 2\nu/\epsilon$, $\ell_0 = -2\nu/\epsilon$.
		\For{$t=1,2,\ldots,T$}{}
		\State Set $u_{t} = u_{t-1} + \delta_u$ and $\ell_t = \ell_{t-1} + \delta_\ell$
		\State \If{there exists $j\in [m]$ and $\alpha > 0$ such that the following three properties hold
		\begin{itemize}
			\item[] (i) $\ell_t\,e \nd_K z_{t-1} + \alpha x_j \nd_{K} u_t e$
			\item[] (ii) $\Phi^{u_t,e}(z_{t-1}+\alpha x_j) \leq \Phi^{u_{t-1},e}(z_{t-1})$
			\item[] (iii) $\Phi_{\ell_t,e}(z_{t-1}+\alpha x_j) \leq \Phi_{\ell_{t-1},e}(z_{t-1})$
		\end{itemize}\hspace{0.5cm}}
		\State Set $z_t = z_{t-1} + \alpha x_j$ and $[y_t]_i = \begin{cases} [y_{t-1}]_i & \textup{if $i\neq j$}\\
		[y_{t-1}]_i + \alpha & \textup{if $i=j$}\end{cases}$
		\Else
		\State \Return fail
		\EndIf
	\EndFor
	    \State \Return $\lambda = \frac{1}{T}y_T$
    \end{algorithmic}
    \label{alg:gen-bss}
\end{algorithm}
The next result tells us that as long as it is possible to simultaneously satisfy the conditions required to take a step at each iteration, then the sparsification algorithm works. 
\begin{lemma}
	\label{lem:gen-bss-outok}
	If the generalized BSS algorithm (Algorithm~\ref{alg:gen-bss}) for a proper 
	cone $K$ terminates correctly (i.e., does not return `fail') for all inputs, 
	then $\sparse_K(\epsilon) \leq \lceil 4\nu/\epsilon^2\rceil$.
\end{lemma}

\begin{proof}
	It is enough to show that, for an arbitrary input, if the 
	generalized BSS algorithm  does not return `fail' then  
	it returns $\lambda \geq 0$ with at 
	most $T$ non-zero entries that satisfies 
	$(1-\epsilon)e \nsd_K \sum_{i\in [m]}\lambda_i x_i \nsd_K (1+\epsilon)e$.

	Since at most one entry of $y$ becomes non-zero at each iteration, it follows 
	that $y_T$ (and hence $\lambda$) has at most $T$ non-zero entries. 
	It is straightforward to check (by induction on $t$) that the iterates of the algorithm have the property 
	\[ \sum_{i\in [m]} [y_t]_ix_i = z_t,\quad\textup{for all $t\in [T]$}.\]
	Therefore $\sum_{i\in [m]}\lambda_i x_i = \frac{1}{T}z_T$.
	By construction $\frac{1}{T}z_T$ satisfies $(\ell_T/T)e \nd_K \frac{1}{T}z_T \nd_K (u_T/T)e$. Furthermore, 
	\[ \frac{\ell_T}{T} = \frac{\ell_0 + T\delta_\ell}{T} = -\frac{2\nu/\epsilon}{\lceil4 \nu/\epsilon^2\rceil} + 1-\epsilon/2 \geq 1-\epsilon,\]
	where the inequality follows from $\lceil x \rceil \geq x$. 
	Similarly 
	\[ \frac{u_T}{T} = \frac{u_0 + T\delta_u}{T} = \frac{2\nu/\epsilon}{\lceil 4 \nu/\epsilon^2\rceil} + 1+\epsilon/2 \leq 1+\epsilon.\]
	It follows that 
	\[ (1-\epsilon)e \nsd_K (\ell_T/T)e \nd_K  \sum_{i\in [m]}\lambda_i x_i \nd_K (u_T/T)e \nsd_K (1+\epsilon)e.\]
\end{proof}
To establish Theorem~\ref{thm:main-sc-gen}, it remains to establish that the generalized BSS algorithm always terminates correctly.

The following result gives a sufficient condition relating the value of the barriers ($\epsilon_u$ and $\epsilon_\ell$) and the size of the barrier shifts ($\delta_u$ and $\delta_\ell$), such that it is possible to take a step, $\alpha x_j$ and ensure that the shifted barriers do not increase in value. 

\begin{proposition}
	\label{prop:bss-prog}
	Let $x_1,\ldots,x_m\in K\setminus\{0\}$ be such that $\sum_{i=1}^{m}x_i = e\in \interior{K}$. Let $\phi$ be a $\nu$-logarithmically homogeneous pairwise-self-concordant barrier for $K$ with 
	associated upper and lower barrier functions defined in~\eqref{eq:ubarrier-def} and~\eqref{eq:lbarrier-def}.
	Let $u,\ell,\delta_\ell,\delta_u,\epsilon_u,\epsilon_\ell > 0$. Let $\ell\,e \nd_K x \nd_K u\,e$ 
	and let $\Phi^{u,e}(x) \leq \epsilon_u$ and $\Phi_{\ell,e}(x) \leq \epsilon_\ell$.
	If $0<\delta_u^{-1} + \epsilon_u \leq \delta_\ell^{-1} - \epsilon_\ell$ then 
	there exists $j\in [m]$  and $\alpha>0$ such that
	\begin{itemize}
		\item[] (i) $(\ell+\delta_\ell) e\nd_K x+\alpha x_j \nd_K (u+\delta_u)e$
		\item[] (ii) $\Phi^{u+\delta_u,e}(x+\alpha x_j) \leq \Phi^{u,e}(x)$
		\item[] (iii) $\Phi_{\ell+\delta_\ell,e}(x+\alpha x_j) \leq \Phi_{\ell,e}(x)$
	\end{itemize}
\end{proposition}
The proof of this result, given in Section~\ref{sec:barrier-prog}, builds on a number of intermediate properties of pairwise-self-concordant barriers.  
With this result established, it is relatively straightforward to show that the generalized BSS 
algorithm always makes progress. 
\begin{proposition}
	\label{prop:correct-termination}
	If $K$ is a proper convex cone and $\phi$ is a $\nu$-logarithmically homogeneous pairwise-self-concordant barrier for $K$ then the generalized BSS algorithm (Algorithm~\ref{alg:gen-bss}) terminates correctly
	for any valid input.
\end{proposition}
\begin{proof}
	We will to show that, at every iteration $t$, we can find $\alpha>0$ and $j\in [m]$ such that
	(i)--(iii) are satisfied. 
	We argue by strong induction on $t$. 

	For the base case, when $t=1$, we know that
	\[ \Phi^{u_0,e}(z_0) = -D\phi(u_0 e)[e] = u_0^{-1}\nu = \epsilon/2,\]
	where the second equality holds because $\phi$ is logarithmically homogeneous, so $D\phi$ is positively homogeneous of degree $-1$. Similarly
	\[ \Phi_{\ell_0,e}(z_0) = -D\phi(-\ell_0 e)[e] = (-\ell_0)^{-1}\nu = \epsilon/2.\]
	Furthermore, $\ell_0 e = -(2\nu/\epsilon)e \nd_K z_0 \nd_K (2\nu/\epsilon)e = u_0 e$. 

	We can apply Proposition~\ref{prop:bss-prog} with $x = z_0 = 0$, $u = u_0 = 2\nu/\epsilon$, $\ell = \ell_0 = -2\nu/\epsilon$, $\epsilon_u = \epsilon_\ell = \epsilon/2$, and $\delta_u = 1+\epsilon/2$ and $\delta_\ell = 1-\epsilon/2$. Note that 
	\begin{equation}
		\label{eq:delta-epsilon}
		0<\delta_u^{-1} + \epsilon_u = \frac{1 + (\epsilon/2) + (\epsilon/2)^2}{1+(\epsilon/2)} \leq \frac{1-(\epsilon/2)+(\epsilon/2)^2}{1-(\epsilon/2)} = \delta_\ell^{-1}-\epsilon_\ell
	\end{equation}
	where the inequality holds because $(1-t)(1+t+t^2) = 1-t^3\leq 1+t^3 =(1+t)(1-t+t^2)$ for all $t>0$.

	Proposition~\ref{prop:bss-prog} tells us that there exists $\alpha>0$ and $j\in [m]$ such that $\ell_1 e \nd_K z_0 + \alpha x_j \nd_K u_1 e$ and 
	$\Phi^{u_1,e}(z_0 + \alpha x_j) \leq \Phi^{u_0,e}(z_0)$ and 
	$\Phi_{\ell_1,e}(z_0+\alpha x_j) \leq \Phi_{\ell_0,e}(z_0)$. 
	So conditions (i)---(iii) are satisfied.

	Suppose that at every iteration $1\leq s\leq t-1$ we can find $\alpha>0$ and $j\in [m]$ 
	such that (i)---(iii) are satisfied. Then $\ell_{t-1}e \nd_K z_{t-1} \nd_K u_{t-1}e$
	and $\Phi^{u_{t-1},e}(z_{t-1}) \leq \Phi^{u_0,e}(z_0) = \epsilon/2$ and 
	 $\Phi_{\ell_{t-1},e}(z_{t-1}) \leq \Phi_{\ell_0,e}(z_0) = \epsilon/2$. 

	 We again apply Proposition~\ref{prop:bss-prog} with $x = z_{t-1}$, $u=u_{t-1}$, $\ell = \ell_{t-1}$, $\epsilon_u = \epsilon_\ell = \epsilon/2$, and $\delta_u = 1+\epsilon/2$ and $\delta_\ell = 1-\epsilon/2$. From~\eqref{eq:delta-epsilon} we have that $0< \delta_u^{-1}+\epsilon_u \leq \delta_{\ell}^{-1} - \epsilon_\ell$. Therefore, there exists $\alpha>0$ and $j\in [m]$ such that
	 $\ell_t\,e\nd_K z_{t-1} + \alpha x_j \nd_K u_t\,e$ and 
	 $\Phi^{u_t,e}(z_{t-1}+\alpha z_j) \leq \Phi^{u_{t-1},e}(z_{t-1}) \leq \epsilon/2$ and
	 $\Phi_{\ell_t,e}(z_{t-1}+\alpha x_j) \leq \Phi_{\ell_{t-1},e}(z_{t-1}) \leq \epsilon/2$. 
	 Therefore conditions (i)---(iii) are satisfied and we can 
	 conclude that the algorithm never returns `fail', as required.
\end{proof}
We note that Theorem~\ref{thm:main-sc-gen} follows directly from Lemma~\ref{lem:gen-bss-outok} and Proposition~\ref{prop:correct-termination}. Therefore, it remains to prove Proposition~\ref{prop:bss-prog}.

\subsection{Properties of pairwise-self-concordant barriers}
	In this section we summarize the key properties of the Hessian of pairwise-self-concordant 
	barriers that we will use in what follows.
	The first tells us about properties of the Hessian inner product of two points in the cone
	under the assumption that the barrier is pairwise-self-concordant.
\begin{lemma}
	\label{lem:nc-hess}
	Let $\phi:\interior{K} \rightarrow \RR$ be a pairwise-self-concordant barrier for a closed, convex, pointed, full-dimensional cone $K$.
	If $x\in \interior{K}$ and $u,v\in K$ then $D^2\phi(x)[u,v]\geq 0$. If, in addition, $u\in \interior{K}$ and $v\in K\setminus\{0\}$ then $D^2\phi(x)[u,v]>0$.
\end{lemma}
\begin{proof}
	First consider the case when $|u|_x = 0$. Then $u\in K \cap (-K)$ and so, since $K$ is pointed, $u = 0$. In this case $D^2\phi(x)[u,v] = 0$. Otherwise consider the case that $|u|_x>0$. In this case, the definition of a pairwise-self-concordant barrier tells us that 
	$D^2\phi(x)[u,v] \geq  0$. 

	If $u\in \interior{K}$ then there exists some $\eta>0$ such that $u - \eta v\in K$. Therefore $D^2\phi(x)[u-\eta v,v]\geq 0$. Since $D^2\phi(x)[\cdot,v]$ is linear, it follows that
	\[ D^2\phi(x)[u,v] \geq \eta D^2\phi(x)[v,v] = \eta\|v\|_x^2 > 0,\]
	where the final inequality holds because $v\neq 0$ and $K$ is pointed.
\end{proof}
The other technical result we will need controls how the Hessian inner products between
elements of the cone change from point to point. The argument is similar to typical 
ways to control how the Hessian norm changes for (ordinary) self-concordant functions.

\begin{lemma}
	\label{lem:ssc-hessian}
	Let $\phi:\interior{K}\rightarrow \RR$ be a pairwise-self-concordant barrier for a closed, convex, pointed, full-dimensional cone $K$.
	Let  $x\in\interior{K}$ and $u,v\in K$ be arbitrary. Then 
	\begin{align}
		D^2\phi(x+tv)[u,v] & \geq \frac{1}{(1+t|v|_x)^2}D^2\phi(x)[u,v]\;\;\textup{for all $t\geq 0$}\\
		D^2\phi(x-tv)[u,v] & \leq \frac{1}{(1-t|v|_x)^2}D^2\phi(x)[u,v]\;\;\textup{for all $0\leq t<1/|v|_x$}.
	\end{align}
\end{lemma}
\begin{proof}
	Consider the univariate functions $g(t) = D^2\phi(x-tv)[u,v]$ defined for $t\in (-\infty,1/|v|_x)$
	and $h(t) = D^2\phi(x+tv)[u,v]$ defined for $t\in (-1/|v|_x,\infty)$. 
	We have that
	\[ g'(t) = -D^3\phi(x-tv)[u,v,v]\quad\textup{and}\quad h'(t) = D^3\phi(x+tv)[u,v,v].\]
	By the pairwise-self-concordance assumption
	\begin{align*}
	 0\leq 	g'(t) & \leq 2D^2\phi(x-tv)[u,v]|v|_{x-tv} = 2g(t)|v|_{x-tv}\quad\textup{and}\\
	0\geq 	h'(t) & \geq -2D^2\phi(x+tv)[u,v]|v|_{x+tv} = -2h(t)|v|_{x+tv}.
	\end{align*}
	Since $v\in K$, 
	\[|v|_{x-tv} 
		= \inf\{\tau\;:\; v \nsd_K \tau(x-tv)\}\quad\textup{and}\quad
		|v|_{x+tv}  = \inf\{\tau\;:\; v \nsd_K \tau(x+tv)\}.\]
	Now $v \nsd_K |v|_x x$ implies that $(1-|v|_xt)v \nsd_K |v|_x(x-tv)$, which implies that (for $t<1/|v|_x$)
	\[v \nsd_K \frac{|v|_x}{1-t|v|_x}(x-tv).\]
	Clearly, then, $|v|_{x-tv} \leq  \frac{|v|_x}{1-t|v|_x}$. A similar argument shows that $|v|_{x+tv} \leq \frac{|v|_x}{1+t|v|_x}$. 
	From Lemma~\ref{lem:nc-hess}, $g(t)\geq 0$ and $h(t)\geq 0$. Therefore 
	\begin{equation}
		\label{eq:ghprimebounds}
	g'(t) \leq 2g(t)\frac{|v|_x}{1-t|v|_x}\quad\textup{and}\quad h'(t) \geq -2h(t)\frac{|v|_x}{1+t|v|_x}.\end{equation}
	Then for $0\leq t<1/|v|_x$, using~\eqref{eq:ghprimebounds} gives 
	\[ \frac{d}{dt}\left[(1-t|v|_x)^2g(t)\right] = -2|v|_x(1-t|v|_x)g(t) + (1-t|v|_x)^2g'(t) 
	\leq -2|v|_x(1-t|v|_x)g(t) + 2|v|_x(1-t|v|_x)g(t) =0.\]
	Therefore, $(1-t|v|_x)^2g(t)$ is monotonically non-increasing so
	\[ (1-t|v|_x)^2D^2\phi(x-tv)[u,v] = (1-t|v|_x)^2g(t) \leq g(0) = D^2\phi(x)[u,v].\]
	Similarly, for $t\geq 0$, 
	\[ \frac{d}{dt}\left[(1+t|v|_x)^2h(t)\right] = 2|v|_x(1+t|v|_x)h(t) + (1+t|v|_x)^2h'(t) \geq 2|v|_x(1+t|v|_x)h(t) - 2|v|_x(1+t)|x|_v)h(t)=0.\]
	Therefore, $(1+t|v|_x)^2h(t)$ is monotonically non-decreasing so
	\[ (1+t|v|_x)^2D^2\phi(x+tv)[u,v] = (1+t|v|_x)^2h(t) \geq h(0) =  D^2\phi(x)[u,v].\]
\end{proof}

\subsection{Barrier progress results}
\label{sec:barrier-prog}

The aim of this section is to establish Proposition~\ref{prop:bss-prog} (stated in Section~\ref{sec:bss-alg}) which shows that it is always possible to make progress in the generalized BSS algorithm.
This result follows from key intermediate results that control how the upper and lower barriers change (Propositions~\ref{prop:uprog} and~\ref{prop:lprog}) as well as how to take steps 
that ensure both barriers remain under control~\ref{prop:bss-prog2}. The section concludes with a proof of Proposition~\ref{prop:bss-prog}. 

Propositions~\ref{prop:uprog},~\ref{prop:lprog}, and~\ref{prop:bss-prog} play the 
role of the original Lemmas 3.3, 3.4, and 3.5 from~\cite{batson2012twice}. They are, however, not quite 
direct generalizations, since they are reformulated to streamline certain aspects of the proof (in particular, to completely avoid Claim 3.6 of~\cite{batson2012twice}). 

Before stating and proving the barrier progress lemmas (Propositions~\ref{prop:uprog},~\ref{prop:lprog}, and~\ref{prop:bss-prog2}), we establish some preliminary results
that control how the upper and lower barriers change with changes in their arguments. 
\begin{lemma}
	\label{lem:eu-bound}
	Let $\phi$ be a pairwise-self-concordant barrier for a proper cone $K$. 
	Let $e\in \interior{K}$, let $u>0$, let $x$ be such that $ue-x\pd_K 0$, and let $u' > u$. 
	Then 
	\[ \Phi^{u',e}(x) - \Phi^{u,e}(x) = -D\phi(u'e-x)[e] - (-D\phi(ue - x)[e]) \leq - (u'-u)D^2\phi(u'e-x)[e,e].\]
\end{lemma}
\begin{proof}
	Consider the univariate function $g:[u,\infty)\rightarrow \RR$ defined by 
	$g(t) = -D\phi(t e - x)[e]$. Note that $g'(t) = -D^2\phi(te-x)[e,e])$ and $g''(t) = -D^3\phi(te-x)[e,e,e]$. Since $\phi$ is pairwise-self-concordant, $g$ is a convex function.
	By convexity of $g$,
	\[ 0\leq  g(u) - [g(u') + g'(u')(u-u')] = -D\phi(ue - x)[e] + D\phi(u'e - x)[e] - D^2\phi(u'e-x)[e,e](u'-u).\]
	Rearranging gives the desired inequality.
\end{proof}

\begin{lemma}
	\label{lem:el-bound}
	Let $\phi$ be a pairwise-self-concordant barrier for a proper cone $K$.
	Let $e\in \interior{K}$, let $\ell>0$, let $x$ be such that $x-\ell' e\pd_K 0$ and let $\ell'>\ell$. Then
	\[ \Phi_{\ell',e}(x) - \Phi_{\ell,e}(x) = -D\phi(x-\ell'e)[e] - (-D\phi(x-\ell e)[e]) \leq (\ell'-\ell) D^2\phi(x-\ell' e)[e,e].\]
\end{lemma}
\begin{proof}
	Consider the univariate function $g:(-\infty,\ell']\rightarrow \RR$ defined by 
	$g(t) = -D\phi(x-t e)[e]$. Note that $g'(t) = D^2\phi(x-te)[e,e]$ and $g''(t) = -D^3\phi(x-te)[e,e,e]$. Since $\phi$ is pairwise-self-concordant, $g$ is a convex function.
	By convexity of $g$,
	\[ 0\leq  g(\ell) - [g(\ell') + g'(\ell')(\ell-\ell')] = -D\phi(x-\ell e)[e] + D\phi(x-\ell'e)[e] + D^2\phi(x-\ell'e)[e,e](\ell'-\ell).\]
	Rearranging gives the desired inequality.
\end{proof}

 \begin{proposition}[Upper barrier shift]
 	\label{prop:uprog}
 	Let $\phi$ be a pairwise-self-concordant barrier for a proper convex cone $K$.
 	Let $e\in \interior{K}$, let $u' = u+\delta_u>u>0$, let $x$ be such that 
   $ue-x\pd_K 0$, and let 
 	$\Delta\in K \setminus\{0\}$. If 
 	\begin{equation}
 		\label{eq:uprog}
 	1 \geq \frac{D^2\phi(u'e-x)[\Delta,e]}{\delta_u D^2\phi(u'e-x)[e,e]} - D\phi(u'e-x)[\Delta]=:U[\Delta]\end{equation}
 	then 
		 $u'e - (x+\Delta) \pd_K 0$ and
		$-D\phi(u'e-(x+\Delta))[e] \leq -D\phi(ue - x)[e]$.
 \end{proposition}
\begin{proof}
	For the first conclusion, we note that $ue - x\pd_K 0$ implies that $u'e - x \pd_K 0$ since $e\in \interior{K}$. Since $K$ is pointed, $D^2\phi(u'e-x)[e,e]>0$. Since $e\in \interior{K}$, by Lemma~\ref{lem:nc-hess} we have that $D^2\phi(u'e-x)[\Delta,e]>0$.
	Therefore~\eqref{eq:uprog} implies that $-D\phi(u'e-x)[\Delta] < 1$. 
	Then, from item (iv) of Lemma~\ref{lem:lhsc-prop}, we know that $|\Delta|_{u'e-x}<1$, or equivalently, that
	$-(u'e-x) \nd_K \Delta \nd_K u'e-x$. Rearranging gives $u'e-(x+\Delta) \pd_K 0$.

	For the second conclusion, we write
	\begin{multline}
	\label{eq:bss-abs-bnd}
	-D\phi(u'e - (x+\Delta))[e] +D\phi(ue-x)[e] = \left[-D\phi(u'e-x)[e] +D\phi(ue-x)[e]\right]\\ + \left[-D\phi(u'e-(x+\Delta))[e] + D\phi(u'e-x)[e]\right]\end{multline}
	and upper bound the two terms in~\eqref{eq:bss-abs-bnd} separately. 
	Lemma~\ref{lem:eu-bound} tells us that 
	\[ -D\phi(u'e-x)[e] - (-D\phi(ue - x)[e]) \leq -\delta_u D^2\phi(u'e-x)[e,e].\]
	For the second term in~\eqref{eq:bss-abs-bnd} we use Taylor's theorem with integral remainder together with the Hessian control bound (Lemma~\ref{lem:ssc-hessian}) for $\phi$. This gives
	\begin{align*}
		-D\phi(u'e-(x+\Delta))[e] + D\phi(u'e-x)[e]& = \int_{0}^{1}D^2\phi(u'e-x-t\Delta)[e,\Delta]\;dt\\
		& \leq \int_{0}^1 \frac{1}{(1-t|\Delta|_{u'e-x})^2}D^2\phi(u'e-x)[e,\Delta]\;dt\\
		&= \frac{D^2\phi(u'e-x)[e,\Delta]}{1-|\Delta|_{u'e-x}}\\
		& \leq \frac{D^2\phi(u'e-x)[e,\Delta]}{1-(-D\phi(u'e-x)[\Delta])}\\
		& \leq \delta_uD^2\phi(u'e-x)[e,e],
	\end{align*}
	where the second last inequality holds because $|\Delta|_{u'e-x} \leq -D\phi(u'e-x)[\Delta]$ and the last inequality follows from the assumption that $U[\Delta]\leq 1$. 
	Substituting into~\eqref{eq:bss-abs-bnd} we see that $-D\phi(u'e-(x+\Delta))[e] + D\phi(ue-x)[e] \leq 0$, as required.
\end{proof}

\begin{proposition}[Lower barrier shift]
	\label{prop:lprog}
	Let $\phi$ be a pairwise-self-concordant barrier for a proper cone $K$.
	Let $e\in \interior{K}$, let $\ell' = \ell+\delta_\ell > \ell > 0$, let $x$ be such that 
	$x-\ell e\pd_K 0$, and let $\Delta \in K\setminus\{0\}$. If $\Phi_{\ell,e}(x) = -D\phi(x-\ell e)[e] < \delta_\ell^{-1}$ and
	\begin{equation}
		\label{eq:lprog}
	1 \leq \frac{D^2\phi(x-\ell'e)[\Delta,e]}{\delta_\ell D^2\phi(x-\ell'e)[e,e]} + D\phi(x-\ell'e)[\Delta] =: L[\Delta]\end{equation}
	then 
		$x+\Delta - \ell'e \pd_K 0$ and
		$-D\phi(x+\Delta - \ell'e)[e] \leq -D\phi(x-\ell e)[e]$.
\end{proposition}
\begin{proof}
	For the first conclusion, we note that $-D\phi(x-\ell e)[e] < \delta_\ell^{-1}$ implies that 
	$|e|_{x-\ell e} < \delta_\ell^{-1}$ (by part (iv) of Lemma~\ref{lem:lhsc-prop}). 
	This means that $-\delta_\ell^{-1}(x-\ell e) \nd_K e \nd_K \delta_{\ell}^{-1}(x-\ell e)$. Rearranging we see that $x-\ell e - \delta_\ell e \pd_K 0$. It follows that $x-\ell' e \pd_K 0$. 

	For the second conclusion, we write
	\begin{multline}
	\label{eq:bss-abs-bnd-l}
	-D\phi(x+\Delta - \ell'e)[e] +D\phi(x-\ell e)[e] = \left[-D\phi(x-\ell'e)[e] +D\phi(x-\ell e)[e]\right]\\ + \left[-D\phi(x+\Delta - \ell'e)[e] + D\phi(x-\ell'e)[e]\right]\end{multline}
	and bound the two terms in~\eqref{eq:bss-abs-bnd-l} separately. 
	Lemma~\ref{lem:el-bound} tells us that 
	\[ -D\phi(x-\ell'e)[e] - (-D\phi(x-\ell e)[e]) \leq \delta_\ell D^2\phi(x-\ell'e)[e,e].\]
	For the second term in~\eqref{eq:bss-abs-bnd-l} we use Taylor's theorem with integral remainder together with the Hessian control bound (Lemma~\ref{lem:ssc-hessian}) for $\phi$. This gives
	\begin{align*}
		-D\phi(x+\Delta - \ell'e)[e] + D\phi(x-\ell'e)[e]& = -\int_{0}^{1}D^2\phi(x+t\Delta - \ell'e)[e,\Delta]\;dt\\
		& \leq -\int_{0}^1 \frac{1}{(1+t|\Delta|_{x-\ell'e})^2}D^2\phi(x-\ell'e)[e,\Delta]\;dt\\
		& = -\frac{D^2\phi(x-\ell'e)[e,\Delta]}{1+|\Delta|_{x-\ell'e}}\\
		& \leq -\frac{D^2\phi(x-\ell'e)[e,\Delta]}{1+(-D\phi(x-\ell'e)[\Delta])}\\
		& \leq -\delta_\ell D^2\phi(x-\ell'e)[e,e],
	\end{align*}
	where the second last inequality follows from the fact that $|\Delta|_{x-\ell'e} \leq -D\phi(x-\ell'e)[\Delta]$ and the last inequality follows from the assumption that $L[\Delta]\geq 1$. 
	Substituting into~\eqref{eq:bss-abs-bnd-l} we see that $-D\phi(x+\Delta - \ell'e)[e] + D\phi(x-\ell e)[e] \leq 0$, as required.
\end{proof}

We conclude this section with a proof of the third key progress result used to prove the correctness of the generalized BSS algorithm. 
 \begin{proposition}
 	\label{prop:bss-prog2}
 	Let $x_1,\ldots,x_m\in K\setminus\{0\}$ be such that $\sum_{i=1}^{m}x_i = e\in \interior{K}$. 
 	Let $\delta_\ell,\delta_u,\epsilon_u,\epsilon_\ell > 0$. Let $\ell\,e \nd_K x \nd_K u\,e$ 
 	and let $-D\phi(u e - x)[e] \leq \epsilon_u$ and $-D\phi(x-\ell e)[e] \leq \epsilon_\ell$.
	 Let $U[\cdot]$ and $L[\cdot]$ denote the linear functionals defined in~\eqref{eq:uprog} and~\eqref{eq:lprog}, respectively. 
	 If $0<\delta_u^{-1} + \epsilon_u \leq \delta_\ell^{-1} - \epsilon_\ell$ then 
	 there exists $j\in [m]$  and $\alpha>0$ such that $U[\alpha x_j] \leq 1 \leq L[\alpha x_j]$ 
 \end{proposition}
\begin{proof}
	It is enough to show that 
	\[ L[e] = \sum_{j\in [m]} L[x_j] \geq \sum_{j\in [m]} U[x_j] = U[e].\]
	If this were the case then 
	\[ \sum_{j\in [m]} (L[x_j] - U[x_j]) \geq 0\]
	from which we can deduce that there exists some $j$ such that $L[x_j] \geq U[x_j]$. Furthermore, since $x_j\in K\setminus\{0\}$ and $e\in \interior{K}$, it follows from the definition~\eqref{eq:uprog} of $U[\cdot]$ and from Lemma~\ref{lem:nc-hess} that $U[x_j]> 0$.  
	Choosing $\alpha = \frac{2}{L[x_j] + U[x_j]}>0$ would complete the proof because
	\[ \frac{2}{L[x_j] + U[x_j]} U[x_j] \leq 1\quad\iff\quad L[x_j] - U[x_j]\geq 0\]
	and 
	\[ \frac{2}{L[x_j]+U[x_j]} L[x_j] \geq 1\quad\iff\quad L[x_j] - U[x_j] \geq 0.\]

	It remains to show that $L[e] \geq U[e]$. To do so, we observe that the univariate function $g(t) = -D\phi(te-x)[e]$ is monotonically non-increasing and the univariate function $h(t) = -D\phi(x-te)[e]$ is monotonically non-decreasing. This is because $g'(t) = -D^2\phi(te-x)[e,e] \leq 0$ and $h'(t) = D^2\phi(x-te)[e,e]\geq 0$. It then follows that 
	\[ L[e] = -(-D\phi(x-\ell'e)[e]) + \delta_\ell^{-1} \geq -(-D\phi(x-\ell e)[e]) +\delta_\ell^{-1} \geq -\epsilon_\ell + \delta_{\ell}^{-1}\]
	and 
	\[ U[e] = -D{\phi}(u'e-x)[e] + \delta_u^{-1} \leq -D\phi(ue -x)[e] + \delta_u^{-1} \leq \epsilon_u + \delta_u^{-1}.\]
	The inequality $L[e]\geq U[e]$ then follows from our assumption that $ -\epsilon_\ell + \delta_\ell^{-1} \geq \epsilon_u+\delta_u^{-1} $.
\end{proof}
We conclude this section with a proof of Proposition~\ref{prop:bss-prog}. This completes all of the arguments that constitute the proof of Theorem~\ref{thm:main-sc-gen}.
\begin{proof}[{Proof of Proposition~\ref{prop:bss-prog}}]
	Proposition~\ref{prop:bss-prog2} tells us that there exists $j\in [m]$ and $\alpha>0$ such that $U[\alpha x_j] \leq 1$.
	From the upper barrier shift result (Proposition~\ref{prop:uprog}), we can conclude that 
	$x+\alpha x_j \nd_K u + \delta_u$ and that $-D\phi((u+\delta_u)e - (x+\alpha x_j))[e] \leq -D\phi(ue - x)[e]$.

	Similarly, Proposition~\ref{prop:bss-prog2} tells us that $L[\alpha x_j] \geq 1$. Moreover, $\Phi_{\ell,e}(x) = -D\phi(x-\ell e)[e] \leq \epsilon_\ell < \delta_{\ell}^{-1}$ (from the assumption that $\delta_{\ell}^{-1} - \epsilon_\ell> 0$). From the lower barrier shift result (Proposition~\ref{prop:lprog}), we can conclude that $(\ell+\delta_\ell)e \nd_K x+\alpha x_j$ and that $-D\phi((x+\alpha x_j) - (\ell+\delta_\ell)e)[e] \leq -D\phi(x-\ell e)[e]$.
\end{proof}

\section{Pairwise-self-concordant barriers}
\label{sec:ssc}

In this section, we establish a sufficient condition for a self-concordant barrier $\phi$ for a convex cone $K$ to be pairwise-self-concordant, i.e., to satisfy Definition~\ref{def:ssc}.
The sufficient condition is based on the operator monotonicity of a family of univariate functions associated with $\phi$. 

\begin{lemma}
	\label{lem:mono-cond}
	Let $\phi$ be a self-concordant barrier for a closed, pointed, full-dimensional 
	convex cone $K$. Given $x\in \interior{K}$, and $u,v\in K$ with $|u|_x\leq 1$, define
	$h:(-1,\infty)\rightarrow \RR$ by 
	\[ h(t) = D\phi(x+tu)[v].\]
	If $h$ is operator monotone for all choices of $x\in \textup{int}(K)$, and $u,v\in K$ with 
	$|u|_x\leq 1$ then $\phi$ is a pairwise-self-concordant barrier for $K$.
\end{lemma}

\begin{proof}
	We make an affine change of variables 
	and consider $g:(0,\infty)\rightarrow \RR$ defined by 
	$g(t) = h(t-1)$.
	This has the following properties:
	\begin{itemize}
		\item $g(t) = D\phi(x+(t-1)u)[v]$; 
		\item $g'(t) = D^2\phi(x+(t-1)u)[v,u]$ and so $g'(1) = D^2\phi(x)[v,u]$;
		\item $g''(t) = D^3\phi(x+(t-1)u)[v,u,u]$ and so $g''(1) = D^3\phi(x)[v,u,u]$. 
	\end{itemize}
	By the integral representation~\eqref{eq:op-mono-int} of operator monotone functions defined on the positive half-line (from Section~\ref{sec:prelim-opmono}), there is a positive measure supported on $[0,1]$ such that 
	\[ g(t) = g(1) + \int_{0}^{1}\frac{t- 1}{\lambda(t-1)+1}\;d\mu(\lambda).\]
	Differentiating under the integral sign, we have that, for all $t\in (1-\epsilon,1+\epsilon)$ (with $0<\epsilon<1/2$, say),  
	\begin{equation*}
		g'(t)  = \int_0^1 \frac{1}{(\lambda(t-1)+1)^2}\;d\mu(\lambda)\;\;\textup{and}\;\;
		g''(t)  = \int_0^1 \frac{-2\lambda}{(\lambda(t-1)+1)^3}\;d\mu(\lambda).
	\end{equation*}
	These formulas are valid by applying the dominated convergence theorem. Indeed, in each case the integrand is uniformly bounded because, for all $\lambda\in [0,1]$ and all $t\in (1-\epsilon,1+\epsilon)$ we have that $|t-1| \leq \epsilon$ and $|\lambda(t-1)+t|\geq |t|-\lambda|t-1| \geq 1-2\epsilon$. Therefore,  $\left|\frac{t-1}{\lambda(t-1)+t}\right| \leq \frac{\epsilon}{1-2\epsilon}$ and $\left|\frac{1}{(\lambda(t-1)+t)^2}\right| \leq \frac{1}{(1-2\epsilon)^2}$.

	Now $g'(1) = \int_0^1\;d\mu(\lambda) = D^2\phi(x)[v,u]$. 
	The inequality
	\[ 0 \leq -D^3\phi(x)[v,u,u] \leq 2D^2\phi(x)[v,u]\]
	follows from the observing that 
	$-g''(1) = 2\int_0^1\lambda\;d\mu(\lambda)$, 
	and $0\leq \int_0^1\lambda\;d\mu(\lambda)\leq \int_0^1\;d\mu(\lambda) = g'(1)$.

	To establish that $\phi$ is pairwise-self-concordant, 
	we need to show that $0\leq -D^3\phi(x)[v,u,u] \leq 2D^2\phi(x)[v,u]|u|_x$ 
	for all $u\in K$, not just for $u\in K$ such that $|u|_x\leq 1$. 
	For a general (unnormalized) $u\in K$, we know that 
	\[ 0\leq -D^3\phi(x)[v,u/|u|_x,u/|u|_x] \leq 2D^2\phi(x)[v,u/|u|_x].\]
	The result follows by multiplying through by $|u|_x^2$. 
\end{proof}

It is fairly straightforward to check that the sufficient condition of Lemma~\ref{lem:mono-cond} is 
satisfied by the self-concordant barrier $-\log\det(X)$ for the positive semidefinite cone.
\begin{lemma}
	\label{lem:psd-mono}
	Let $X \pd 0$ be any positive definite matrix, and $U,V \psd 0$ be positive
	semidefinite matrices and assume that $U\nsd X$. 
	Let $\phi(X) = -\log\det(X)$. 
	Consider the function $h:(-1,\infty)\rightarrow \RR$ defined by
	\[ h(t) = D\phi(X+tU)[V] = -\textup{tr}\left[(X+t U)^{-1}V\right].\]
	Then $h$ is operator monotone.
\end{lemma}
\begin{proof}
	For simplicity of notation let $U_X = X^{-1/2}UX^{-1/2}$ and let $V_X = X^{-1/2}VX^{-1/2}$. 
	Since $U$ and $V$ are positive semidefinite, it follows that $U_X$ and $V_X$ 
	are also positive semidefinite. Using this notation, then 
	\[ h(t) = -\textup{tr}[(I+tU_X)^{-1}V_X].\]
	Let $U_X = \sum_{i=1}^{d} \lambda_i P_i$ be 
	an eigenvalue decomposition of $U_X$ where $P_i$ is the (positive semidefinite) 
	orthogonal projector onto 
	the eigenspace corresponding to the eigenvalue $\lambda_i$. 
	Since $U_X$ is positive semidefinite, 
	each of the $\lambda_i$ is positive.
	Then 
	\[ h(t) = -\sum_{i=1}^{d}\frac{1}{1+t \lambda_i} \tr(P_i V_X).\]
	Each of the functions $-1/(1+ta)$ for $a>0$ is operator monotone. 
	Since $\tr(P_iV_X) \geq 0$ for all $i$, we have expressed  
	 $h$ as a non-negative combination of such functions, and so $h$ is 
	operator monotone.
\end{proof}

The same result holds for any hyperbolic barrier function. This can be established 
directly using similar techniques to those employed to show that hyperbolic barriers have the 
negative curvature property~\cite[Theorem 6.1]{guler1997hyperbolic}. Another approach is
to invoke the Helton-Vinnikov theorem\footnote{Note that the result in~\cite{helton2007linear} is stated in an alternative, but equivalent, dehomogenized form. For a statement similar to the one
given here, see~\cite{lewis2005lax}.}, which says that any three-variable 
hyperbolic polynomial can be expressed 
in terms of the determinant restricted to a three-dimensional subspace of 
real symmetric matrices that contains the identity matrix.
\begin{theorem}[{\cite{helton2007linear}}]
Let $p(x,y,z)$ be hyperbolic with respect to $(1,0,0)\in \RR^3$. Then there exists a positive constant $c$ and real symmetric matrices $A$ and $B$, such that 
\[ p(x,y,z) = c\det(xI + yA + zB)\quad\textup{for all $(x,y,z)\in \RR^3$}.\]
\end{theorem}
One interpretation of this result is that questions about hyperbolic polynomials and hyperbolicity cones that only depend on the restriction to a three-dimensional subspace can be reduced to questions about real symmetric matrices and the positive semidefinite cone.

\begin{corollary}
	\label{cor:hyp-mono}
	Let $p$ be a polynomial that is homogeneous of degree $d$ and hyperbolic with respect to $e$. 
	Then $-\log p$ is a pairwise-self-concordant barrier for $\Lambda_+(p,e)$.
\end{corollary}
\begin{proof}
	Let $x\in \interior{\Lambda_+(p,e)}$ and $u,v\in \Lambda_+(p,e)$ be such that $x-u \in \Lambda_+(p,e)$. Let $q(w,y,z) = p(wx + yu + zv)$ and note that $q$ is hyperbolic with respect to $(1,0,0)$. By the Helton-Vinnikov theorem there is a positive constant $c$ and 
	real symmetric matrices $U,V\in \cS^d$ such that $q(w,y,z) = c\det(wI + yU+zV)$. 
	Moreover the assumptions on $U$ and $V$ imply that $V \psd 0$ and $0\nsd U \nsd I$. 
	Then $h(t) = D\phi(x+tu)[v] = -c\textup{tr}[(I+tU)^{-1}V]$. It follows from Lemma~\ref{lem:psd-mono} that $h$ is operator monotone, and hence from Lemma~\ref{lem:mono-cond} that $-\log p$ is a pairwise-self-concordant barrier.
\end{proof}

\section{Further results on sparsification functions}
\label{sec:further}

In this section, we discuss further bounds on sparsification functions of cones that can 
be obtained by combining results from Sections~\ref{sec:basics},~\ref{sec:sc}, and~\ref{sec:hyp-sparse}. 

\subsection{Cones having lifted representations with respect to hyperbolicity cones}
\label{sec:hyp-lift-bounds}

We begin with the proof of Corollary~\ref{cor:hyp-lift}, which says that if $p$ is hyperbolic with respect to $e$ and $C$ is a closed convex cone with a $\Lambda_+(p,e)$-lift, then $\sparse_C(\epsilon) \leq \lceil 4\textup{deg}(p)/\epsilon^2\rceil $ for all $\epsilon\in (0,1)$. 
\begin{proof}[{Proof of Corollary~\ref{cor:hyp-lift}}]
	If $C$ has a proper $\Lambda_+(p,e)$-lift, then $\sparse_C(\epsilon) \leq \sparse_{\Lambda_+(p,e)}(\epsilon)$ for all $\epsilon$ (by Proposition~\ref{prop:lifts}), and so the bound follows from the bound on the sparsification function for a hyperbolicity cone (Corollary~\ref{thm:hyp-sparse}).

	As such, assume that that the lift is not proper. In other words, suppose that $C = \pi(\Lambda_+(p,e)\cap L)$ where $\pi$ is a linear map and $L$ a linear space that does not meet the relative interior of $\Lambda_+(p,e)$. Let $F$ be the smallest face of $\Lambda_+(p,e)$ that contains $\Lambda_+(p,e)\cap L$. Then it follows, for instance, from~\cite[Proposition 2.2]{lourenco2022amenable} that 
	$\Lambda_+(p,e)\cap L = F \cap L$ and that $\relint(F) \cap L$ is non-empty. 
	Furthermore, it follows from~\cite[Corollary 3.4]{lourencco2024hyperbolicity} that $F$ is the hyperbolicity cone corresponding to some hyperbolic polynomial $q$ with $\textup{deg}(q) \leq \textup{deg}(p)$. Therefore, $C = \pi(F \cap L)$ has a proper $F$-lift, and so 
	\[ \sparse_C(\epsilon) \leq \sparse_F(\epsilon) \leq \lceil 4\,\textup{deg}(q)/\epsilon^2\rceil \leq \lceil 4\,\textup{deg}(p)/\epsilon^2\rceil.\]
\end{proof}

\subsection{Sparsification functions and dual cones}
\label{sec:sc-egs}

Recall that Theorem~\ref{thm:sc-sparse} gives a bound of $\lceil(4\nu/\epsilon)^2\rceil $ on the sparsification function of any convex cone that admits a $\nu$-logarithmically homogeneous self-concordant barrier.
We have seen that this result can be improved to $\lceil4\nu/\epsilon^2\rceil $ when $K$ admits a $\nu$-logarithmically homogeneous pairwise-self-concordant barrier. However, there are situations where
Theorem~\ref{thm:sc-sparse} gives interesting sparsification results.

For instance, suppose that $K = \Lambda_+(p,e)^*$ is the dual cone of a hyperbolicity cone associated with a low-degree hyperbolic polynomial $p$. 
In general, the dual cone of a hyperbolicity cone is not a hyperbolicity cone~\cite{ramana1995some}, and little is known about lifted representations of these cones.
Despite this, we can get a logarithmically homogeneous self-concordant barrier for $\Lambda_+(p,e)^*$ with parameter $\textup{deg}(p)$ via the following result.
\begin{theorem}[{\cite[Theorem 2.4.4]{nesterov1994interior}}]
	\label{thm:conj-barrier}
	Let $K$ be a closed pointed full-dimensional convex cone and let $\phi$ be a $\nu$-logarithmically homogeneous self-concordant barrier for $K$. 
	Then $\phi^*(y):= \sup_{x} \langle x,y\rangle - \phi(x)$ is a $\nu$-logarithmically homogeneous self-concordant barrier for $-K^*$. 
\end{theorem}
Applying Theorem~\ref{thm:sc-sparse} immediately yields the following bound on the sparsification functions 
of the dual cones of hyperbolicity cones.
\begin{corollary}
	\label{cor:hyp-dual}
	Suppose that $p$ is a hyperbolic polynomial of degree $d$, hyperbolic with 
	respect to $e$ with associated 
	hyperbolicity cone $\Lambda_+$. Let $\Lambda_+^*$ denote the associated dual cone. 
	Then $\sparse_{\Lambda_+^*}(\epsilon) \leq \lceil (4d/\epsilon)^2\rceil $ for all $\epsilon\in (0,1)$.
\end{corollary}
This is remarkable because it gives another source of convex cones, other than hyperbolicity cones, where the sparsification function is bounded above by a quantity that is independent of the dimension of the cone.

We currently do not know, in general, whether the dual cone of a hyperbolicity cone admits a pairwise-self-concordant barrier, let alone one that is $\nu$-logarithmically homogeneous for some $\nu < 4\deg(p)^2$. As such, it is unclear whether  there is any way to obtain an improved bound on $\sparse_{\Lambda_+^*}(\epsilon)$ by directly applying Theorem~\ref{thm:main-sc-gen}.

\subsection{Example: epigraphs of the spectral and nuclear norms}
As an illustration of the results in this section, let $n > k$ be positive integers and consider the cones
\begin{align*}
	K_{\textup{sp}} & = \left\{(X,t)\in \RR^{n\times k}\times \RR\;:\; \sigma_{1}(X) \leq t\right\}\quad\textup{and}\quad\\
	K_{\textup{nuc}} & = \left\{(X,t)\in \RR^{n\times k}\times \RR\;:\; \sum_{i=1}^{k}\sigma_i(X) \leq t\right\},
\end{align*}
where $\sigma_{1}(X)\geq \sigma_2(X) \geq \cdots \geq \sigma_k(X)$ are the singular values of the $n\times k$ matrix $X$ (with $k\leq n$). 
The cone $K_{\textup{sp}}$ is the epigraph of the spectral norm. The cone $K_{\textup{nuc}}$ is the epigraph of the nuclear norm. Consider the setting where $k$ is fixed and $n$ is growing. The special case $k=1$ gives the second order cone (or Lorentz cone). 

It turns out that $K_{\textup{nuc}} = K_{\textup{sp}}^*$, which follows from the duality of the spectral and nuclear norms. Moreover, $K_{\textup{sp}}$ is the hyperbolicity cone associated with the hyperbolic polynomial $-\log\det(t^2I_k - X^\intercal X)$ of degree $2k$~\cite[Section 6]{bauschke2001hyperbolic}. Therefore the sparsification function of $K_{\textup{sp}}$ is bounded above by $\lceil 8k/\epsilon^2\rceil$, despite being a cone of dimension $nk+1$. 

Interestingly, $K_{\textup{sp}}$ also admits the $(k+1)$-logarithmically homogeneous self-concordant barrier, 
\[  -\log\det(tI_k - X^\intercal X/t) - \log(t)\]
(see, e.g.,~\cite[Proposition 5.4.6]{nesterov1994interior}), which is not a hyperbolic barrier. However, this means that $K_{\textup{nuc}}$ also admits a $(k+1)$-logarithmically homogeneous self-concordant barrier (by Theorem~\ref{thm:conj-barrier}). Theorem~\ref{thm:sc-sparse}
then implies that the sparsification function of 
$K_{\textup{nuc}}$ is bounded above by $\lceil (4(k+1)/\epsilon)^2\rceil $, despite having dimension $nk+1$.

\section{Sparsification and packing and covering conic programs}
\label{sec:consequences}

This section considers a special class of conic programs that generalize packing and covering linear programs. These are called covering and packing conic programs, respectively. Sparsification gives structural information about near-optimal solutions of covering conic programs. Sparsification also allows packing conic programs to be approximated by a potentially simpler problem instances.

A \emph{covering linear program} is an optimization problem of the form
\begin{equation}
\label{eq:LP-c} \min_{y\in \RR_+^k}\; b^\intercal y\;\;\textup{subject to}\;\; \sum_{i=1}^{m}y_ia_i\geq c.\end{equation}
where $c\in \RR_+^d$,  $a_i\in \RR_+^d$ for $i=1,2,\ldots,n$, and $b\in \RR_+^k$. 
The corresponding dual problem, known as a \emph{packing linear program}, has the form
\begin{equation}
\label{eq:LP-p} \max_{x \in \RR_+^d}\; c^\intercal x\;\;\textup{subject to}\;\; a_i^\intercal x \leq b_i\;\;\textup{for $i=1,2,\ldots,k$}\end{equation}

This section considers how natural conic generalizations of~\eqref{eq:LP-c} and~\eqref{eq:LP-p} can be sparsified in certain ways. 
Let $\cE$ be a finite dimensional real inner product space.
Let $K\subseteq \cE$ be a closed convex cone and let $K^*\subseteq \cE^*$ be the corresponding dual cone. Let $a_1,\ldots,a_m\in K$ and $c\in K$. 
A natural conic generalization of a covering linear program
is an optimization problem of the form
\begin{equation}\label{eq:conic-c} \textsc{cover}(b,c) := \inf_{y\in \RR_+^k}\;\langle b,y\rangle\;\;\textup{subject to}\;\; \sum_{i=1}^{k}y_ia_i \psd_{K} c,\end{equation}
where $b\in \RR_+^k$.
The corresponding dual problem has the form
\begin{equation}\label{eq:conic-p} \textsc{pack}(b,c) := \sup_{x\in K^*}\;\langle c,x\rangle \;\;\textup{subject to}\;\;
 \langle a_i,x\rangle \leq b_i\;\;\textup{for $i=1,2,\ldots,k$},\end{equation}
and is a natural conic generalization of a packing linear program. If each entry of 
$b$ is strictly positive then it is straightforward to check that 
strong duality holds for this pair of conic programs since Slater's condition holds.

The following result shows that if we replace the cost vector $c$ of~\eqref{eq:conic-p} with a cost 
vector $c'$ that is close to $c$ with respect to the ordering induced by $K$ (e.g., obtained via sparsification if $c$ is expressed as a sum of elements of $K$), 
then the optimal value of the resulting problem is approximately preserved. 
This idea is used in the context of algorithms for the max-cut semidefinite program where the cost is the Laplacian of a graph, and the modified cost is obtained via spectral sparsification.
\begin{proposition}
Suppose that $c\in \relint(K)$ and $c = \sum_{i=1}^{m}c_i$ where $c_1,c_2,\ldots,c_m\in K$. Let $c'$ be such that $(1-\epsilon)c \nsd_K c'\nsd_K (1+\epsilon)c$.  If $b>0$ then  
\[(1-\epsilon)\textsc{pack}(b,c) \leq \textsc{pack}(b,c') \leq (1+\epsilon)\textsc{pack}(b,c).\]
\end{proposition}
\begin{proof}
Let $z$ be any feasible point for~\eqref{eq:conic-p}.
Then, since $z\in K^*$, we have that
\[ (1-\epsilon)\langle c,z\rangle \leq \langle c',z\rangle \leq (1+\epsilon)\langle c,z\rangle.\]
But then 
\[ \textsc{pack}(b,c') = \sup_{x\in K^*}\langle c',x\rangle \geq \langle c',z\rangle \geq \langle c,z\rangle (1-\epsilon).\]
Since $z$ is an arbitrary point that satisfies 
$z\in K^*$ and  $\langle a_i,z\rangle \leq b_i$ for $i=1,2,\ldots,k$, we can take the supremum of the right hand side over $z$ to establish that $\textsc{pack}(b,c') \geq (1-\epsilon)\textsc{pack}(b,c)$. 

For the reverse inequality, let $y\in \RR_+^k$ be such that $\sum_{i=1}^{k}y_ia_i \psd_K c$ be any feasible point for the dual problem~\eqref{eq:conic-c}. 
Then $(1+\epsilon)y\in \RR_+^k$ and  
	satisfies
\[\sum_{i=1}^{k} (1+\epsilon)y_ia_i = (1+\epsilon)\sum_{i=1}^{k}y_ia_i \psd_K (1+\epsilon)c \psd_K c'.\]
Therefore, $(1+\epsilon)\langle b,y\rangle \geq \textsc{cover}(c')$. Taking the infimum over $y\in \RR_+^k$ such that $\sum_{i=1}^{k}y_ia_i \psd_K c$ gives
\[ (1+\epsilon)\textsc{cover}(b,c) \geq \textsc{cover}(b,c').\]
	By strong duality we have that $\textsc{cover}(b,c')= \textsc{pack}(b,c')$ and $\textsc{cover}(b,c) = \textsc{pack}(b,c)$, completing the argument. 
\end{proof}

Sparsification can be used in another way in the context of conic covering problems. Indeed, if 
non-trivial sparsification with respect to $K$ is possible, then it implies the existence of near-optimal sparse solutions to conic covering problems~\eqref{eq:conic-c}.
\begin{proposition}
	Suppose that $b>0$ and $c\in \relint(K)$ and let $y\in \RR_+^k$ be an optimal point for~\eqref{eq:conic-c}. Then, for any $\epsilon \in (0,1)$, there exists $y'\in \RR_+^k$ such that
	\begin{itemize}
		\item $y'$ is feasible for~\eqref{eq:conic-c};
		\item $y'$ has at most $\sparse_K(\epsilon)$ non-zero entries; and
		\item $\frac{1-\epsilon}{1+\epsilon}\langle b,y'\rangle \leq \textsc{cover}(b,c) \leq \langle b,y'\rangle$
	\end{itemize}
\end{proposition}
\begin{proof}
	Consider the point $e = \sum_{i=1}^{k}y_i a_i \psd_K c$. Since $c\in \relint(K)$ it follows that $e\in \relint(K)$. Fix $\epsilon\in (0,1)$. 
	From the definition of the sparsification function of $K$, there exists $S\subseteq[k]$ 
	with $|S|\leq  \sparse_K(\epsilon)$, and positive 
	scalars $(\lambda_i)_{i\in S}$, such that 
	\begin{equation}
		\label{eq:sp-cover}
		(1-\epsilon)e \nsd_K \sum_{i\in S}\lambda_i y_i a_i \nsd_K (1+\epsilon)e.
	\end{equation}
	Let $y_i' = \lambda_i y_i/(1-\epsilon)$ for $i\in S$ and $y_i' = 0$ for $i\in [k]\setminus S$. 
	Then $y_i'\geq 0$ for all $i$ and $\sum_{i=1}^{k}y_i'a_i = \sum_{i\in S} \lambda_i y_i a_i/(1-\epsilon) \psd_K e \psd_K c$. Therefore $y'$ is feasible for~\eqref{eq:conic-c}. It follows that $\langle b,y'\rangle \geq \textsc{cover}(b,c)$.  

	Let $x$ be any optimal point for~\eqref{eq:conic-p}. Since $b>0$, strong duality holds. Therefore, the primal-dual optimal pair $(y,x)$ satisfies the complementarity conditions 
	\begin{align}
		\langle c,x\rangle &= \left\langle \sum_{i=1}^{k}y_ia_i,x\right\rangle = \langle e,x\rangle\quad\textup{and}\label{eq:cs-cone}\\
	y_ib_i &= y_i\langle a_i,x\rangle\quad\textup{for all $i=1,2,\ldots,k$}.\label{eq:cs-cone2}
	\end{align} 
	We then have that 
	\begin{equation*}  \langle c,x\rangle \stackrel{(a)}{=} \left\langle e,x\right\rangle
	\stackrel{(b)}{\geq} \frac{1}{1+\epsilon}\left\langle \sum_{i\in S}\lambda_i y_i a_i,x\right\rangle\\ = \frac{1}{1+\epsilon}\sum_{i\in S}\lambda_iy_i\langle a_i,x\rangle \stackrel{(c)}{=} \frac{1}{1+\epsilon} \sum_{i\in S}\lambda_i y_i b_i = \frac{1-\epsilon}{1+\epsilon}\langle b,y'\rangle\end{equation*}
	where (a) follows from~\eqref{eq:cs-cone}, (b) follows from~\eqref{eq:sp-cover} and the fact that $x\in K^*$, and (c) follows from~\eqref{eq:cs-cone2}. 
	The final result then follows 
	from the fact that $\textsc{cover}(b,c) = \textsc{pack}(b,c) = \langle c,x\rangle$. 

\end{proof}

\section{Discussion}
\label{sec:discussion}

This paper introduces the sparsification function of a closed convex cone, and develops basic bounds on this quantity for different families of convex cones. Theorem~\ref{thm:bss} directly generalizes the celebrated linear-sized spectral sparsification results of Batson, Spielman, and Srivastava~\cite{batson2012twice} to the setting of convex cones that admit a pairwise-self-concordant barrier, such as the hyperbolic barriers for hyperbolicity cones. We conclude by discussing some natural questions raised by the results in this paper.

\paragraph{Algorithms:} Our proof of Theorem~\ref{thm:main-sc-gen}, which gives a bound on the sparsification function of convex cones that admit a pairwise-self-concordant barrier, is algorithmic. However, the resulting algorithm, while polynomial time in the case of the positive semidefinite cone, 
is not particularly efficient. There has been a line of work~\cite{allen2015spectral,lee2017sdp,jambulapati2024linear,lau2025spectral} that gives faster algorithms for linear-sized spectral sparsification, i.e., for the case of sparsifying sums of elements of the positive semidefinite cone. It would be interesting to understand the extent to which these algorithms can be generalized to the case of a convex cone that admits a 
$\nu$-logarithmically homogeneous pairwise-self-concordant barrier. 

Similarly, another approach to sparsification (for cuts~\cite{benczur1996approximating}, hypergraphs~\cite{soma2019spectral}, sums of norms~\cite{jambulapati2023sparsifying}, linear codes~\cite{khanna2024code}, and in the spectral sparsification model~\cite{spielman2008graph}) is based on importance sampling. The rough idea is that sampling the terms in the sum based on suitably defined importance weights gives an appropriately sparsified sum, with high probability. In the case of spectral sparsification, such randomized constructions come at the expense of a logarithmic factor in the number of terms required. It would be interesting to see whether a natural analogue of such randomized constructions can achieve, up to logarithmic factors, the same sparsification results as we achieve here. The main challenge appears to be controlling the resulting empirical process over the dual cone $K^*$.

\paragraph{Sparsification with constrained weights:} In the present paper, the 
only requirement on the weights $\lambda_i$ (for $i\in [m]$) is that they are non-negative. There
are interesting graph sparsification problems (for instance degree-preserving sparsification)
where it is natural also to ask that the weights also satisfy certain constraints, such as
 that they lie in a subspace. The discrepancy-based approach to spectral
sparsification is flexible enough to sparsify sums of positive semidefinite matrices while 
ensuring that the weights lie in a given (high-dimensional) affine subspace~\cite{lau2025spectral}.
It would be interesting to extend our abstract sparsification model to also capture this situation.

\paragraph{Pairwise-self-concordant barriers:} We have seen strong bounds on the sparsification
function of convex cones that admit pairwise-self-concordant barriers. However, the most general family 
of convex cones we are aware of that admit pairwise-self-concordant barriers are hyperbolicity cones. 
It would be very interesting to construct other convex cones that admit such barriers. It
seems likely that this is possible, given that hyperbolicity cones enjoy the (apparently)
stronger property that $t\mapsto D\phi(x+tv)[u]$ is operator monotone whenever $x\in \interior{K}$ and $u,v\in K$. It would be also interesting to understand how this property of self-concordant barriers relates to other properties, such as the negative curvature property~\cite{nesterov2016local}.

\paragraph{Faces:} We have seen, in Section~\ref{sec:faces}, that if $F$ is a face of a convex cone $K$ then, under certain additional technical assumptions on the face, $\sparse_{F}(\epsilon)\leq \sparse_K(\epsilon)$ for all $\epsilon\in (0,1)$. It would be interesting to determine whether these additional technical assumptions are necessary or whether sparsification functions are monotone (or even strictly monotone) along faces, in general. The main challenge here appears to be that of relating errors with respect to the ordering induced by $K$ to errors 
with respect to the ordering induced by $F$. 

\paragraph{Duality:} Our most general bound on the sparsification function of a cone (Theorem~\ref{thm:sc-sparse}) is invariant under duality, in the sense that it gives the same bound on the sparsification function of $K$ and $K^*$ (by using the fact that if $K$ has a $\nu$-logarithmically homogeneous self-concordant barrier, then so does $K^*$). However, we do not know, in general how the sparsification functions of $K$ and $K^*$ are related.

\section*{Acknowledgments}
This research was supported in part by an Australian Research Council Discovery Early Career Researcher Award (project number DE210101056) funded by the Australian Government. This work also benefited from the support and stimulating environment provided by the MATRIX institute during the workshop \emph{Theory and applications of stable polynomials} held in 2022. 

\bibliographystyle{alpha}
\bibliography{refs-hypsp}

\end{document}